\documentclass[preprint,11pt]{elsarticle}
\usepackage{subcaption}
\usepackage{amsmath}
\usepackage{amsthm}
\usepackage{multirow}
\usepackage{tikz}
\usepackage{soul}
\usepackage{tikz-cd}
\usepackage{lipsum}
\usepackage{import}

\usepackage{amsfonts}
\usepackage{graphics}
\usepackage[
letterpaper, top=1in, bottom=1in, left=1in, right=1in]{geometry}
\usepackage{xcolor}
\usepackage{amssymb}
\usepackage{psfrag}
\usepackage{graphicx}
\usepackage[utf8]{inputenc}
\usepackage{bm}
\usepackage{mathtools}
\usepackage{epstopdf}
\usepackage[colorlinks,linkcolor=black,anchorcolor=black,citecolor=black]{hyperref}
\usepackage{cleveref}
\usepackage{setspace}
\usepackage{amsmath,amssymb,bbm}
\usepackage[utf8]{inputenc}
\usepackage{url}
\usepackage{color}
\usepackage{booktabs}
\usepackage[linesnumbered]{algorithm2e}
\usepackage{placeins}

\usepackage[title]{appendix}
\usepackage{lmodern}
\usepackage{lineno}

\newtheorem*{Theo}{Theorem}

\makeatletter
\newcommand{\oset}[3][0ex]{%
  \mathrel{\mathop{#3}\limits^{
    \vbox to#1{\kern-2\ex@
    \hbox{$\scriptstyle#2$}\vss}}}}
\makeatother

\newcommand{\dx}{{d_{\cX}}}
\newcommand{\dy}{{d_{\cY}}}

\newcommand{\R}{\mathbb{R}}

\newcommand{\cA}{\mathcal{A}}

\newcommand{\cD}{\mathcal{D}}
\newcommand{\cE}{\mathcal{E}}

\newcommand{\cG}{\mathcal{G}}

\newcommand{\cL}{\mathcal{L}}
\newcommand{\cM}{\mathcal{M}}
\newcommand{\cN}{\mathcal{N}}

\newcommand{\cP}{\mathcal{P}}
\newcommand{\cQ}{\mathcal{Q}}
\newcommand{\cR}{\mathcal{R}}

\newcommand{\cU}{\mathcal{U}}

\newcommand{\cX}{\mathcal{X}}
\newcommand{\cY}{\mathcal{Y}}

\DeclareMathAlphabet\mathbfcal{OMS}{cmsy}{b}{n}
\usepackage{amsmath,amssymb,bbm}
\usepackage[utf8]{inputenc}

\makeatletter
\def\ps@pprintTitle{%
 \let\@oddhead\@empty
 \let\@evenhead\@empty
 \def\@oddfoot{\reset@font\hfil\thepage\hfil}
 \let\@evenfoot\@oddfoot
}
\makeatother
\begin{document}

\begin{frontmatter}

\title{\textbf{Learning Neural Operators on Riemannian Manifolds}}

\author[a,1]{Gengxiang Chen}
\author[b,1]{{Xu Liu}}
\author[a]{Qinglu Meng}
\author[a]{Lu Chen}
\author[a]{Changqing Liu}
\author[a]{{Yingguang Li}\corref{cor1}}
\address[a]{College of Mechanical and Electrical Engineering, Nanjing University of Aeronautics and Astronautics, 210016, Nanjing, China}

\address[b]{School of Mechanical and Power Engineering, Nanjing Tech University, 211816, Nanjing, China}

\cortext[cor1]{Corresponding author: liyingguang@nuaa.edu.cn}
\fntext[fn1]{Gengxiang Chen and Xu Liu contributed equally.}

\begin{abstract}
In Artificial Intelligence (AI) and computational science,  learning the mappings between functions (called operators) defined on complex computational domains is a common theoretical challenge. Recently, Neural Operator emerged as a promising framework with a discretisation-independent model structure to break the fixed-dimension limitation of classical deep learning models. However, existing operator learning methods mainly focus on regular computational domains, and many components of these methods rely on Euclidean structural data. In real-life applications, many operator learning problems are related to complex computational domains such as complex surfaces and solids, which are non-Euclidean and widely referred to as Riemannian manifolds. Here, we report a new concept, Neural Operator on Riemannian Manifolds (NORM), which generalises Neural Operator from being limited to Euclidean spaces to being applicable to Riemannian manifolds, and can learn the mapping between functions defined on any real-life complex geometries, while preserving the discretisation-independent model structure. NORM shifts the function-to-function mapping to finite-dimensional mapping in the Laplacian eigenfunctions' subspace of geometry, and holds universal approximation property in learning operators on Riemannian manifolds even with only one fundamental block. The theoretical and experimental analysis prove that NORM is a significant step forward in operator learning and has the potential to solve complex problems in many fields of applications sharing the same nature and theoretical principle.
\end{abstract}

\begin{keyword}
{Deep learning, Neural operator, Partial differential equations, Riemannian manifold}
\end{keyword}

\end{frontmatter}

\let\thefootnote\relax\footnote{The source code is available at \href{https://github.com/gengxiangc/NORM}{https://github.com/gengxiangc/NORM}.}

\section{Introduction }\label{sec:intro}

Many scientific discoveries and engineering research activities involve the exploration of the intrinsic connection and relationship between functions \cite{wang2023scientific, degrave2022magnetic}. In mathematics, the mapping between two functions is called the operator \cite{li2020neural}. Establishing the operator defined on complex computational domains has been a theoretical challenge \cite{li2022fourier}. One ubiquitous example of operator is the solution operator of Partial Differential Equations (PDEs) \cite{wang2021learning}, which provide the foundational descriptions of many nature laws. Solving PDEs under different parameters, initial and boundary conditions can be regarded as finding the solution operators \cite{chen2018neural, brunton2023machine}. A more practical example is that, for nuclear fusion, establishing the operator that links the input controlling coil voltage to the plasma distribution in the complex tokamak vessel could enable rapid and accurate forecasting of plasma field evolution, thus pointing to a promising direction towards sustainable fusion \cite{gopakumar2023fourier, degrave2022magnetic}. There are also requirements for establishing operators in a wide range of other complex field prediction scenarios, such as predicting the blood flow dynamics of the human body for the purpose of cardiovascular disease diagnosis and treatment \cite{corti2022impact, kissas2020machine}, and predicting the pressure field of an aircraft for fuselage structure optimisation \cite{sabater2022fast, taverniers2022accelerating}. Physical experiments and numerical simulations are commonly used methods for finding the mapping between two functions (i.e. operators) \cite{karniadakis2021physics, wang2023scientific}. Due to the complex process of the underlying operators, especially when involving complex computational domains like tokamak vessels, human organs or aircraft structures, the high computational and experimental costs of these methods are prohibitive for real-world situations \cite{azzizadenesheli2023neural, ramezankhani2021making}.

Artificial Intelligence (AI) techniques recently emerged as a promising paradigm shift for learning operators directly from data \cite{wang2023scientific}. Classical deep learning methods, such as Convolutional Neural Networks (CNNs) and deconvolution techniques, can learn the mapping between discretised picture-like uniform grid data to approximate the operator \cite{chen2019u, wu2023solving}. Graph Neural Networks (GNNs) can represent the computational domain as a graph and then learn the properties of the nodes through message passing \cite{velickovic2017graph, chen2021graph}. However, since the network structure and the parameterisation of CNNs and GNNs heavily depend on the discretisation-resolution of the computational domain \cite{lu2022comprehensive}, the high-dimensional discretisation of the computational domains will bring significant computational burdens to model training, and lead to slow convergence or even divergence when learning general nonlinear operators \cite{you2022nonlocal}. Recently, Neural Operators (NOs), such as DeepONet \cite{lu2021learning} and Fourier Neural Operator (FNO) \cite{li2020fourier}, were proposed as a new deep learning approach that could directly learn mappings between functions on continuous domains with a discretisation-independent model structure (i.e., the parameterisation of the model is independent of the discretisation of the computational domain) \cite{kovachki2021neural}. Despite the significant success of NOs, they mainly focus on learning the mapping between functions defined on regular computational domains (data in the form of a picture-like uniform grid), and many components of these methods rely on Euclidean structural data, for example, Fast Fourier Transform in FNO \cite{li2020fourier}, image convolution layer in U-shaped Neural Operator (UNO) \cite{rahman2022u}, and Wavelet transform in Wavelet Neural Operator (WNO) \cite{gupta2021multiwavelet}. However, real-life applications are more complex and many are in irregular computational domains. Existing NOs often have to convert irregular data to the form as regular uniform grid by coordinate transformation \cite{li2022fourierGeo, gao2021phygeonet} or grid interpolation \cite{lu2022comprehensive, seiler2015resampling}. However, coordinate transformation techniques are normally limited to converting simple 2-dimensional (2D) irregular computational domains due to the poor intrinsic representation \cite{li2022fourierGeo, lu2022comprehensive}, whilst grid interpolation often leads to high-dimensional discretisation and thus brings significant computational burdens to model training, especially for 3-dimensional (3D) computational domains \cite{wu2023solving}. Therefore, existing NOs have limitations in solving operator learning problems of real-life applications with irregular computational domains, including complex surfaces and solids, which are non-Euclidean structural data, and widely referred to as Riemannian manifolds.

This research proposed a deep learning framework with a new concept called Neural Operator on Riemannian Manifolds (NORM), as shown in Fig. \ref{fig_NORM_method}a. NORM could break the limitations of existing NOs and extend the applicability from Euclidean spaces to Riemannian manifolds. NORM can learn the mapping between functions defined on any Riemannian manifolds, including 2D and 3D computational domains, while maintaining a model structure independent of discretisation. Compared with learning operators directly in the Euclidean space, the fundamental blocks of NORM shift the function-to-function mapping to the finite-dimensional mapping in the Laplacian eigenfunctions’ subspace of geometry (Fig. \ref{fig_NORM_method}c). Because Laplacian eigenfunctions have been proven to be the optimal basis for approximating functions on Riemannian manifolds \cite{aflalo2015optimality}, NORM can learn the global geometric information effectively and accurately without increasing the complexity of parameterisation. Besides, we have proved that NORM could hold the universal approximation property even with only one fundamental block. The effectiveness of the proposed framework was demonstrated through several different tasks in science and engineering, including learning solution operators for classical PDEs, composite workpiece deformation prediction and blood flow dynamics prediction.

\begin{figure}[h] 
\centering
\includegraphics[width=\textwidth]{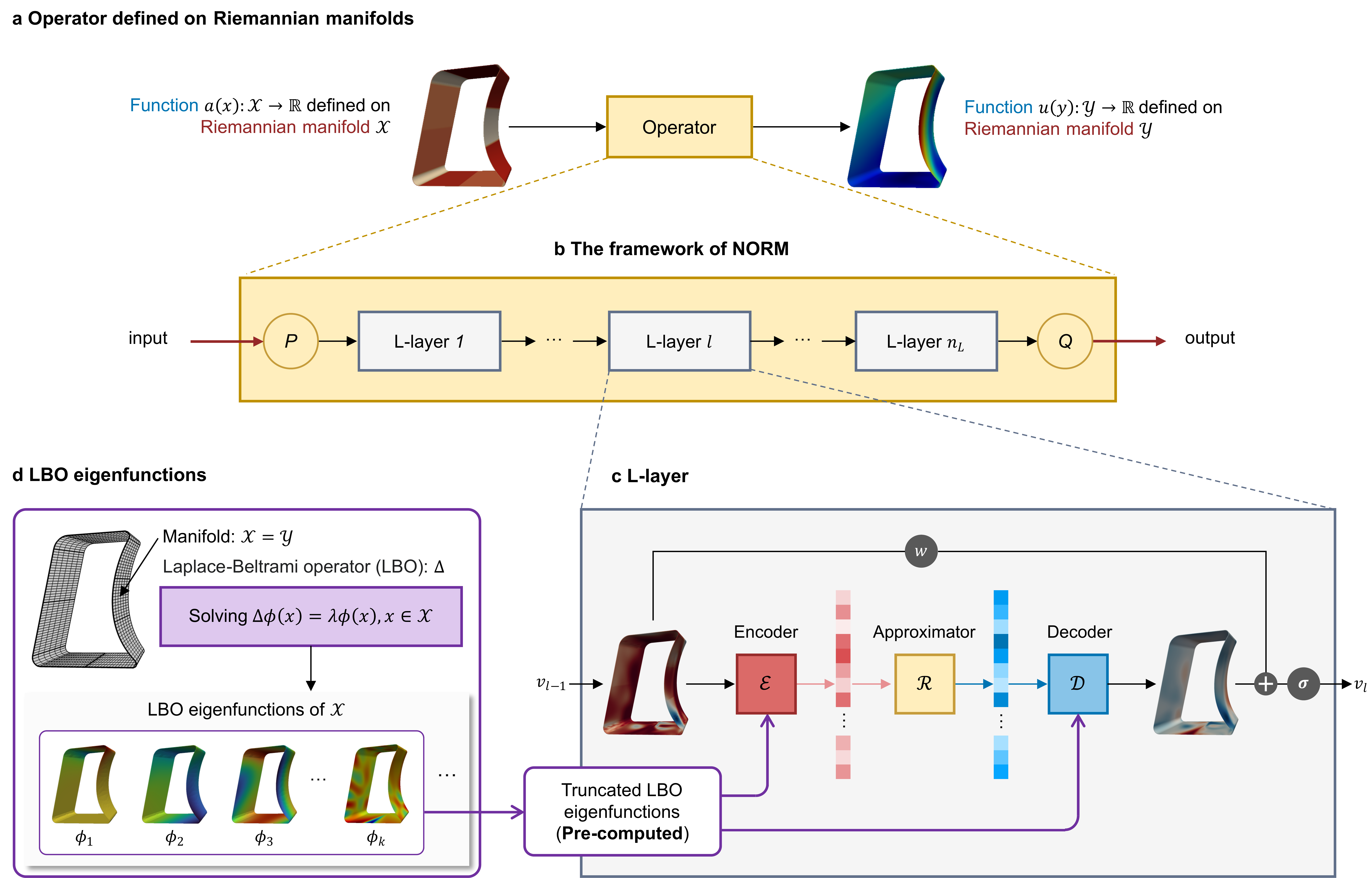}
\caption{\textbf{The illustration of Neural Operator on Riemannian Manifolds (NORM).} \textbf{a}, Operators defined on Riemannian manifolds, where the input function and output function can be defined on the same or different Riemannian manifolds. The example for this illustration is the operator learning problem of the composite curing case, where the input temperature function and the output deformation function are both defined on the same manifold, the composite part. \textbf{b}, The framework of NORM, consists of two feature mapping layers (P and Q) and multiple L-layers. \textbf{c}, The structure of L-layer, consists of the encoder-approximator-decoder block, the linear transformation, and the non-linear activation function. \textbf{d}, Laplace-Beltrami Operator (LBO) eigenfunctions for the geometric domain (the composite part). }
\label{fig_NORM_method}
\end{figure}

\section{Neural operator on Riemannian manifolds}\label{sec:NORM}

\subsection{Problem definition} 

Learning operators on Riemannian manifolds refers to learning a mapping between two functions defined on Riemannian manifolds, as shown in Fig. \ref{fig_NORM_method}a. Denote $\cG: \cA(\cX;\R) \rightarrow \cU(\cY;\R)$ a continuous operator, namely the underlying mapping between the input and the output functions. The input $a \in \cA\left(\cX; \R\right)$ is a function $a(x): \cX \rightarrow \R$, $x \in \cX$, the output $u \in \cU \left(\cY; \R\right)$ is a function $u(y): \cY \rightarrow \R$, $y \in \cY$. $\cX$ and $\cY$ are Riemannian manifolds. Assuming that both $\cA$ and $\cU$ are $L^2$ spaces, then, the problem of learning operator on Riemannian manifolds is to learn a parameterised operator $\cG_{\theta}$ to approximate $\cG$, i.e. $\cG_{\theta} \approx \cG, \theta \in \R^p$. 

Since the input function $a$ and the output function $u$ are both defined on Riemannian manifolds, the obvious solution is to transfer them into a new representation that can be processed with existing Euclidean learning models. Ideally, the solution should be feasible and consistent for any functions defined on Riemannian manifolds. At the same time, the new representation should be low-dimensional while maintaining the information of the original functions. Therefore, we first propose a simple approximation block with an encoder-approximator-decoder structure to transfer the mapping between functions on Riemannian manifolds to a finite-dimensional mapping on Euclidean space.

The approximation block for learning operators on Riemannian manifolds can be defined as a mapping $\cN: \cA(\cX;\R) \rightarrow \cU(\cY;\R)$ of the form $\cN= \cD \circ \cR \circ \cE$, where $\cE:\cA(\cX;\R) \rightarrow\ \R^{\dx}$ denotes the encoder that maps the function on manifold $\cX$ to Euclidean space, $\cR: \R^{\dx} \rightarrow \R^{\dy}$ is an approximator, a learning model for Euclidean data, $\cD:\R^{\dy}\rightarrow\ \cU(\cY;\R)$ is an inverted mapping to recover the prediction function on manifold $\cY$. 

Similar encoder-approximator-decoder structures were also applied in learning mappings between functions defined on Euclidean space \cite{bhattacharya2021model, seidman2022nomad}. To learn operators on Riemannian manifolds, the primary challenge lies in how to design the encoder and decoder mapping to process functions on manifolds without increasing the model complexity. These two mappings would not only influence the feature extraction capability of the learning model, but also determine whether the model holds universal approximation property. 

\subsection{Constructing mappings using Laplacian}
The discretisation-independent target of the neural operator reminds us of the mesh-free spectral method in PDE solving \cite{rai2021spectral}. Intuitively, the spectrum of manifolds could naturally describe the intrinsic information of operators on manifolds. The ideal choice of the spectrum for operator learning is the eigenfunctions of the Laplacian, which is a set of orthonormal basis \cite{patane2018laplacian}, and has been proven to be the optimal basis for approximating functions defined on Riemannian manifolds \cite{aflalo2015optimality, aflalo2013spectral}. Therefore, the encoder $\cE$ and the decoder $\cD$ could be constructed as the spectral decomposition and the spectral reconstruction on the corresponding Laplacian eigenfunctions. 

The Laplacian occurs in a wide range of differential equations describing science and engineering problems, such as the heat transfer function, Poisson's equation, diffusion equation, and wave equation \cite{aflalo2013spectral}. For the Euclidean space $\mathbb{R}^{d}$ and a twice-differentiable function $f$, the Laplacian $\Delta f$ is a second-order differential operator defined as the divergence $\nabla \cdot$ of the gradient $\nabla f$, that is $\Delta f=\nabla \cdot \nabla f=\nabla^2 f$. The eigenvalue problem for the Laplacian can be defined as $\Delta \phi_i =\lambda_i \phi_i$, where the $\lambda_i$ ($\lambda_1 \leq \lambda_2 \leq \ldots$) and $\phi_i(x)$ that satisfying this equation are defined as the eigenvalues and the corresponding eigenfunctions. Actually, the Fourier basis $e^{2 \pi \mathrm{i} k x}$ is also the eigenfunction of the Laplacian with the eigenvalue $\lambda = -(2 \pi k)^2 $ \cite{TERRY_FFT}. 

Since the divergence operator $\nabla \cdot$ and gradient operator $\nabla f$ can also be defined on manifolds with Riemannian metric $g$, the Laplacian $\Delta f$ can be naturally extended to the Riemannian manifold, which is also called the Laplace–Beltrami operator (LBO) \cite{reuter2006laplace}. Therefore, we could obtain the Laplacian spectrum of manifolds in a similar way as in Euclidean space, as shown in Fig. \ref{fig_NORM_method}d.

For Riemannian manifold $\cM$, the LBO eigenfunctions $\phi_i(x)$ is a set of orthonormal bases for the Hilbert space $L^2(\cM)$. It can be proved that a finite number of leading LBO eigenfunctions can approximate functions on manifolds with any accuracy \cite{aflalo2015optimality}. Therefore, for the approximation block $\cN= \cD \circ \cR \circ \cE$, the encoder $\cE$ can be defined as the spectral decomposition on the LBO eigenfunctions $\phi_{\cX,i}$ of the input manifold $\cX$: 
\begin{equation}
\cE:{\cA} \rightarrow\ \R^{\dx}, \quad \cE (a) := (\left\langle a, \phi_{\cX,1} \right\rangle, \ldots, \left\langle a, \phi_{\cX,\dx} \right\rangle), \quad \forall a \in \cA
\end{equation}

And the decoder can be defined as the spectral reconstruction on the LBO eigenfunctions $\phi_{\cY,i}$ of the output manifold $\cY$: 
\begin{equation}
\cD:\R^{\dy}\rightarrow {\cU}, \quad \cD (\beta) = \sum_{i=1}^{\dy} \beta_i \phi_{\cY,i} \quad \forall \, \beta \in \R^{\dy}
\end{equation}

With the defined encoder $\cX$ and decoder $\cD$, an approximation block $\cN= \cD \circ \cR \circ \cE$ could potentially learn the mappings between functions on manifolds with a simple parameterised Euclidean learning model $\cR$. Since LBO can be defined on any Riemannian manifold, the block $\cN$ can naturally deal with any complex geometric domain, which breaks the limitations that existing neural operators relying on Euclidean structured data. Meanwhile, the approximation block holds the discretisation-independent property, because $\cR$ is parameterised on Euclidean spaces with size only related to the truncated eigenfunctions of the input and output manifolds. 

Although Laplacian is defined mathematically on smooth domains, practical numerical computation typically requires discrete approximations of domains, such as meshes or point clouds. The LBOs of common geometric meshes have been strictly defined in the differential geometry field \cite{alexa2020properties}, including triangular mesh, quadrilateral mesh, and tetrahedral mesh. In Supplementary Materials S2.1, an example of discretised LBO for triangular mesh is provided.

\subsection{Framework of Neural Operator on Riemannian Manifold}
The approximation block $\cN= \cD \circ \cR \circ \cE$ can transfer the mapping between functions on Riemannian manifolds to a finite-dimensional Euclidean space learning problem. However, one basic block only approximates the target operator by a linear subspace, which is inefficient in extracting non-linear low-dimensional structures of data. Here, we propose a new deep learning framework, the Neural Operator on Riemannian Manifolds (NORM), that consists of multiple layers and in which the approximation block $\cN$ constitutes one layer of the model, like the convolution layer in traditional CNN. 

We start from a common situation, assuming the input and output functions are defined on the same manifold $\cM$, i.e. $\cX = \cY = \cM$. The structure of NORM can be represented as the form shown in Fig. \ref{fig_NORM_method}b, consisting of two feature mapping layers $\cP$, $\cQ$ and $n_{L}$ hidden layers. The shallow network $\cP$ lifts the input function $a$ to get $v_0 =\cP(a)$, where $\cP: L^2(\cM;\R) \rightarrow L^2(\cM;\R^{d_v})$, so as to expand the dimension of features to increase the representation ability, similar to the convolution channel expansion in CNN. Multiple hidden layers, defined as the Laplace layer, or L-layer (Fig. \ref{fig_NORM_method}c), would update the input function iteratively, such as $v_{l-1} \mapsto v_{l} = \cL_{l}(v_{l-1})$ in the L-layer $l$. After that, the final shallow network $Q$ would project the high-dimensional features to the output dimension, namely $u=Q \left(v_{n_{L}}\right)$, where $\cQ: L^2(\cM;\R^{d_v}) \rightarrow L^2(\cM;\R)$. The iterative structure can be represented as: 

\begin{equation}\label{eq:no_form}
\mathcal{G}_\theta(a)= \cQ\circ \cL_{n_{L}} \circ \cL_{n_{L}-1} \circ \cdots \circ \cL_1 \circ \cP(a)
\end{equation}

The iteration of the hidden layers is given as follows: 
\begin{equation}\label{eq:layer}
v_{l} = \cL_{l}(v_{l-1})(x): =\sigma\left(W_l v_{l-1}(x) + b_l(x) + \cN(v_{l-1})(x)\right), \quad \forall x \in \cM
\end{equation}
where the linear transformations $W_l \in \R^{d_v \times d_v}$ and the bias $b_l \in \R^{d_v}$ are defined as pointwise mapping. $\sigma$ is the non-linear activation function like in the traditional neural network. Note that, the LBO eigenfunctions required in the approximation block can be pre-computated before training the model, as shown in Fig. \ref{fig_NORM_method}d. The detailed implementation of the discredited version of the approximation block $\cN(v)$ is provided in Supplementary Materials S1.1.

The above definition introduces the NORM structure where the input and output are defined on the same manifold. Nevertheless, the structure can be easily generalised to the settings where the input and output are defined on different manifolds and several different structures are introduced in Supplementary Materials S1.2.

Note that the parameterisation of NORM is independent of the discretisation of the input and output functions, because all operations are defined directly in the function spaces on manifolds rather than the Euclidean coordinate spaces. $\cP$, $\cQ$ are learnable neural networks between finite-dimensional Euclidean spaces and have the same point-wise parameterisation for all $x \in \cM$. Therefore, NORM can learn the mappings between functions on any Riemannian manifolds while maintaining the discretisation-independent property. 

\subsection{Universal Approximation of NORM}
Many recent studies have investigated the universal approximation properties of neural operators between functions on Euclidean spaces \cite{lanthaler2022error, kovachki2021neural}. This section will show the advantage of the proposed method that even one approximation block $\cN$ of NORM holds the universal approximation ability in learning operators between functions defined on Riemannian manifolds. 

Let $\cN= \cD \circ \cR \circ \cE$ be a neural operator for the continuous mapping $\cA(\cX;\R) \rightarrow \cU(\cY;\R)$, and $\cR : \R^{\dx} \rightarrow \R^{\dy}$ represents a neural network that has universal approximation property. The encoder is defined as: $ \cE:{\cA} \rightarrow\ \R^{\dx}$ and $\cE (a) := (\left\langle a, \phi_{\cX,1} \right\rangle, \ldots, \left\langle a, \phi_{\cX,\dx} \right\rangle), \forall a \in \cA$. The decoder is defined as $\cD:\R^{\dy}\rightarrow {\cU}$, and $\cD (\beta) = \sum_{i=1}^{\dy} \beta_i \phi_{\cY,i}, \forall \, \beta \in \R^{\dy}$. $\cX$ and $\cY$ are Riemannian manifolds. $\cA$ and $\cU$ are $L^2$ spaces. $\phi_{\cX,i}$ and $\phi_{\cY,i}$ are LBO eigenfunctions of manifolds $\cX$ and $\cY$, respectively. It should be noted that $\cN= \cD \circ \cR \circ \cE$ is a basic block of the NORM, and also can be treated as the simplified version of NORM. Therefore, the universal approximation property of $\cN $ could guarantee the universal approximation property of the more complex NORM framework. The universal approximation theorem of neural operators on Riemannian manifolds is as follows:

\begin{Theo} \textbf{Universal approximation theorem for neural operators on Riemannian manifolds.}
Let $\cG: \cA(\cX;\R) \rightarrow \cU(\cY;\R)$ be a Lipschitz continuous operator, $K \in \cA$ is compact set. Then for any $\epsilon > 0$, there exists a neural operator $\cN= \cD \circ \cR \circ \cE$, such that:
\begin{equation}
\sup _{a \in K}\|\mathcal{G}(a) -\mathcal{N}(a)\|_{L^2} \leq \epsilon
\end{equation}
\label{ua_norm}
\end{Theo} 

\begin{proof}[Proof of Theorem]
It is challenging to directly prove the approximation error from $\cA \rightarrow \cU$. Therefore, we establish a low-dimensional projection subspace of $\cA$ and $\cU$ spanned by the corresponding LBO eigenfunctions. It can be first proved that $\cN$ holds universal approximation property in learning operators between the projection subspaces. Since LBO eigenfunctions is a group of basis in $L^2$ space, the projection error can be proven to be $\epsilon-$approximation. Therefore, the final approximation error of $\cN$ can be obtained by combining the approximation error on the subspace, the encoding error on the input, and the decoding error on the output. The detailed proof can be found in Supplementary Materials S3.

\end{proof}

\section{Results}\label{sec:results}
The proposed NORM was verified using three toy cases and two practical engineering cases with 2D or 3D complex geometric domains. The three toy cases of learning PDEs solution operators involved different problem settings and input/output structures: (1) The Darcy problem case aims to learn the mapping from the parameter function (the diffusion coefficient field) to the solution function (the pressure function field), where both functions are defined on the same 2D manifold; (2) The pipe turbulence case is a classical dynamics systems prediction setting, namely, predicting the future state field based on the current state field (velocity field in the pipe), and (3) The heat transfer case tries to learn the mapping from the boundary condition (temperature function on 2D manifold) to the temperature field of the part (temperature function on 3D manifold). The two engineering cases are composite workpiece deformation prediction and blood flow dynamics prediction: (4) The composite case aims to learn the mapping from the temperature field to the final deformation field of a 3D composite workpiece, where the deformation mechanism involves complex physicochemical processes other than only PDEs, and (5) For the blood flow dynamics case, the inputs are multiple time series functions, and the output is the spatiotemporal velocity field of the aortic (3D manifold).

We compared the NORM with several popular neural operators, including DeepONet \cite{lu2021learning}, POD-DeepONet \cite{lu2022comprehensive}, FNO \cite{li2020fourier}, and also one classical Graph Neural Networks (GNN), named GraphSAGE \cite{hamilton2017inductive}. For the 2D cases, the irregular geometric domains were interpolated to a regular domain for the implementation of FNO. For the 3D cases, we did not compare with FNO because of the prohibitive complexity of 3D spatial interpolation. Since the message-passing mechanism in graph learning methods typically focuses on problems with the same input and output graphs, we did not compare GNN for the heat transfer case and blood flow dynamics case. The details about data generation and baseline model configurations are described in the Supplementary Materials S4 and S5. The quantitative comparison results of all methods are presented in Table. \ref{tab:Composite_result}. We considered two error metrics: $E_{L_2}$ is the mean relative $L_2$ error of all test samples, and Mean Maximum Error (MME) refers to the mean value of all test samples in terms of the maximum error in the whole computational domain. 

\begin{table}[t]%
\scriptsize
\vspace*{-2pt}
\centering
\caption{Performance comparison for the five case studies.}\label{tab2}%
\renewcommand\arraystretch{2}
\begin{tabular*}{0.98\textwidth}{llccccc}
\toprule
\textbf{Case} & \textbf{Metric} & \textbf{GNN} & \textbf{DeepONet} & \textbf{POD-DeepONet} & \textbf{FNO}   & \textbf{NORM} \\
\midrule

\multirow{2}{8em}{\textbf{1.} Darcy problem\\}

& MME  & 0.140(0.002)   & 0.045(4e-4)  & 0.044(3e-4)   & 0.094(0.002) & \textbf{0.039(4e-4) } \\
& $E_{L_2}$(\%) & 6.732(0.053) & 1.358(0.013) & 1.296(0.023)  & 3.826(0.077) & \textbf{1.046(0.020) } \\
\midrule

\multirow{2}{8.5em}{\textbf{2.} Pipe turbulence\\}  
& MME  & 2.358(0.125) & 0.960(0.002) & 0.241(0.017) &  0.896(0.001)&  \textbf{0.116(0.003) } \\
& $E_{L_2}$(\%)   & 23.583(1.411) & 9.358(0.107) & 2.587(0.275)&  3.801(0.002)& \textbf{1.008(0.020) } \\
\midrule

\multirow{2}{8em}{\textbf{3.} Heat transfer\\}  
& MME  & - & 3.038(0.156)   & \textbf{1.304(0.045)} & -  & 1.599(0.096) \\
& $E_{L_2}$(\%)  & - & 0.072(0.002) & 0.057(0.001)   & - & \textbf{0.027(0.002) } \\
\midrule

\multirow{2}{8em}{\textbf{4.} Composite\\}  

& MME   & 0.882(0.029)   & 0.157(0.002) & 0.077(0.003)& - & \textbf{0.051(0.002) } \\
& $E_{L_2}$(\%)   & 20.908(0.050)  & 1.880(0.034) & 1.437(0.060)&  - & \textbf{0.999(0.027)} \\
\midrule

\multirow{2}{8em}{\textbf{5.} Blood flow\\}    
& MME   &-   & 0.899(0.010) & 0.488(0.002) &  - & \textbf{0.093(0.003)}  \\
& $E_{L_2}$(\%)   & -  & 89.260(2.367) & 37.420(0.119) &  - & \textbf{4.822(0.061)}  \\

\bottomrule
\multicolumn{7}{l}{\scriptsize The values A(B) represent the mean and standard deviation of five repeated runs, respectively.} \\
\end{tabular*}\vspace*{-3pt} 
\label{tab:Composite_result}
\end{table}

\subsection{Learning PDEs solution operators}

\begin{figure}[t]
\centering
\includegraphics[width=\textwidth]{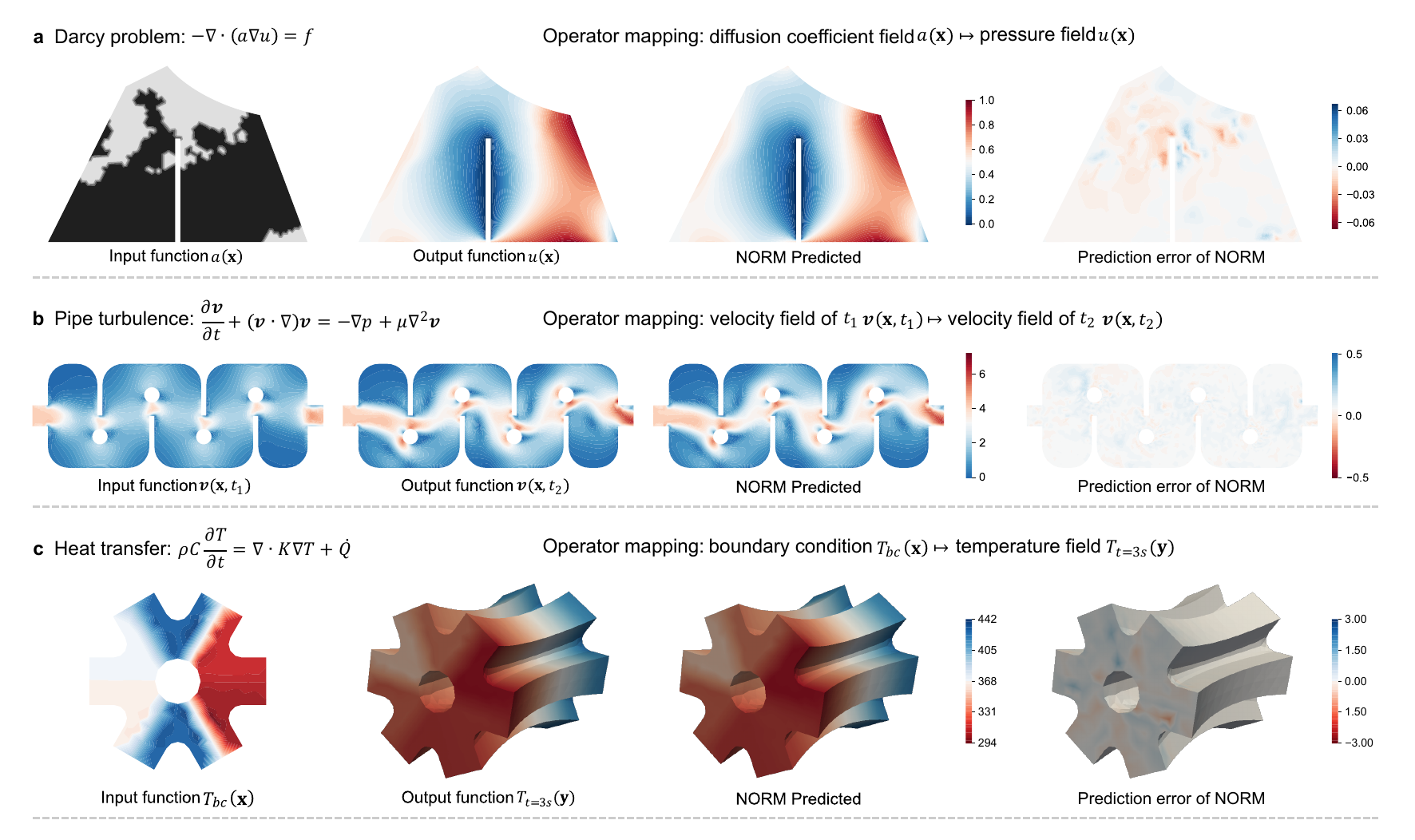}
\caption{\textbf{Illustration of three toy case studies}. \textbf{a}, Darcy problem (Case 1): the operator learning problem is the mapping from the diffusion coefficient field to the pressure field. \textbf{b}, Pipe turbulence (Case 2): the operator learning problem is the mapping from the current velocity field to the future velocity field. \textbf{c}, Heat transfer (Case 3): the operator learning problem is the mapping from the 2D boundary condition to the 3D temperature field of the part.}
\label{fig_three_cases}
\end{figure}

\subsubsection{Darcy problem (Case 1)}
Darcy flow equation is a classical law for describing the flow of a fluid through a porous medium. This problem is also widely used for various neural operator verification \cite{kovachki2021neural}. Darcy's law can be mathematically described by the following equation:

\begin{equation}
\label{eq_darcy}
-\nabla \cdot(a \nabla u )=f
\end{equation}
where $a$ is the diffusion coefficient field, $u$ is the pressure field and $f$ is the source term to be specified. As shown in Fig. \ref{fig_three_cases}a, the computational domain is a 2D geometric shape represented by a triangle mesh with 2290 nodes. The geometric domain has an irregular boundary with a thin rectangle notch inside, which can increase the complexity of the learning problem. The operator learning target in the Darcy flow problem is the mapping from the diffusion coefficient field $a(\mathbf{x})$ to the pressure field $u(\mathbf{x})$:

\begin{equation}
\cG: a(\mathbf{x}) \mapsto u(\mathbf{x}), \quad \mathbf{x} \in \cM
\end{equation}

The labelled data for training the neural operator model is the pair of $a(\mathbf{x})$ and $u(\mathbf{x})$. 1200 sets of input data $a(\mathbf{x})$ are randomly generated first. Then the corresponding $u(\mathbf{x})$ is solved by Matlab's SOLVEPDE toolbox. 1000 of them are used as the training dataset, and the rest 200 groups are defined as the test dataset. 

Fig. \ref{fig_three_cases}a reports the comparative prediction results for one representative in the test dataset. It can be observed that the output field and the NORM predicted result show excellent agreement. The prediction results and errors of comparison methods are provided in the Extended Fig. \ref{darcy_Node2k_2}, in which $\Delta_{mean}$ refers to the average absolute error over all nodes in the geometric domain, and $\Delta_{ max}$ means the maximum absolute error on all nodes. Due to inaccurate grid interpolation, FNO has the most significant error, especially in the boundary region. DeepONet and POD-DeepONet show significant errors on the right side of the rectangle. The quantitative results on the test dataset are listed in Table. \ref{tab:Composite_result}. NORM can achieve the lowest error compared with all other baseline methods. 

\subsubsection{Pipe turbulence (Case 2) }
Turbulence is a vital flow state of the fluid, which reflects the instability of the fluid system \cite{rouse1937modern}. Here, we considered turbulent flows in a complex pipe, of which the underlying governing law is the 2D Navier-Stokes equation for a viscous incompressible fluid:
\begin{equation} \label{eq_pipe}
\begin{aligned}
 \frac{\partial \boldsymbol{v}}{\partial t}+(\boldsymbol{v} \cdot \nabla) \boldsymbol{v} & =-\nabla p+\mu \nabla^2 \boldsymbol{v} \\
\nabla \cdot \boldsymbol{v} &=0 
\end{aligned}
\end{equation}
where $\boldsymbol{v}$ is the velocity, $p$ is the pressure, and the fluid chosen is water. The geometric design of the irregular pipe is shown in Fig. \ref{fig_three_cases}b, where the left and right ends are inlet and outlet, respectively. For a given inlet velocity, we performed the transient simulation to predict the velocity distribution in the pipe. The velocity field data are represented by a triangular mesh with 2673 nodes. Details about data generation and simulation settings can be found in the Supplementary Materials S4.1.2. The operator learning problem of this case is defined as the mapping from the velocity field $\boldsymbol{v}(\mathbf{x}, t_1)$ to the velocity field $\boldsymbol{v}(\mathbf{x}, t_2)$, where $t_2 = t_1+0.1 s$:
\begin{equation}
\cG: \boldsymbol{v}(\mathbf{x}, t_1) \mapsto \boldsymbol{v}(\mathbf{x}, t_2) , \quad \mathbf{x} \in \cM
\end{equation}

The considered baseline methods are the same as Darcy problem. Fig. \ref{fig_three_cases}b shows the predictive performance of NORM on one representative in the test dataset, which gives consistent prediction compared with the ground truth. The prediction results and errors of baseline models are provided in the Extended Fig. \ref{pipe_flow_50_temp}. FNO achieves minor errors in smooth areas but large errors in sharp areas because of the grid interpolation, leading to small $\Delta_{mean}$ but large $\Delta_{ max}$. POD-DeepONet, like NORM, has a uniform distribution of errors, while the error value is slightly larger than NORM. DeepONet has the most significant prediction error compared to other methods in this task. The quantitative statistical results can be seen in Table \ref{tab2}.

\subsubsection{Heat transfer (Case 3)}

Heat transfer describes the transfer of energy as a result of a temperature difference, which widely occurs in nature and engineering technology \cite{li2021transforming}. The heat equation can be written in the following form (assuming no mass transfer or radiation).

\begin{equation}
\label{heat_trans}
\rho C \frac{\partial T}{\partial t}= \nabla \cdot K\nabla T+\dot{Q}
\end{equation}
where $T$ is the temperature as a function of time and space. $\rho$, $C$, and $K$ are the density, specific heat capacity, and thermal conductivity of the medium, respectively. And $\dot{Q}$ is the internal heat source.

The heat transfer case was designed on a three-dimensional solid part, as shown in Fig. \ref{fig_three_cases}c. The learning problem is defined as the mapping from the 2D boundary condition $T_{bc}(\mathbf{x})$ to the 3D temperature field $T_{t=3s} (\mathbf{y})$ of the solid part after 3s of heat transfer.

\begin{equation}
\cG: T_{bc}(\mathbf{x}) \mapsto T_{t=3s}(\mathbf{y}) , \quad \mathbf{x} \in \cX, \mathbf{y} \in \cY, 
\end{equation}

As shown in Fig. \ref{fig_three_cases}c, the input geometric domain is represented by a triangular mesh with 186 nodes, and the output geometric domain is represented by a tetrahedral mesh with 7199 nodes. The labelled dataset was generated by the commercial simulation software Comsol. The training dataset consists of 100 labelled samples, and another 100 groups are defined as the test dataset. More details are given in the Supplementary Materials S4.1.3.

In this case, the input and output functions are defined on different manifolds, thus the different L-layers of NORM have to utilise different LBO eigenfunctions. The model structure of NORM is given in Fig. S1b. The beginning L-layers of NORM employ the LBO eigenfunctions of the input manifold for both the encoder and decoder. One middle L-layer of NORM utilises the LBO eigenfunctions of the input manifold for the encoder while taking the LBO eigenfunctions of the output manifold for the decoder. The ending L-layers employ LBO eigenfunctions of the output manifold for both the encoder and decoder. FNO is not implemented for this case due to the prohibitive computational complexity of 3D spatial interpolation. The prediction results of different methods for one typical test data are shown in the Extended Fig. \ref{heat_transfer_all}. DeepONet has a large prediction error where the temperature gradient is large. POD-DeepONet has different errors on different temperature regions of the left end face, while the error of NORM is smaller and only appears in a few small areas. Moreover, the statistical results for all methods on the test dataset are shown in Table \ref{tab:Composite_result}, where NORM shows the smallest relative $L_2$ error.

\subsection{Composite workpiece deformation prediction (Case 4)}

This case study investigated the effectiveness of the proposed NORM on a complex 3D irregular geometry, specifically in predicting the curing deformation of a Carbon Fiber Reinforced Polymer (CFRP) composite part. CFRP composites are known for their lightweight and high-strength properties, thus becoming preferred materials for weight reduction and performance enhancement in modern aerospace industries \cite{ramezankhani2021making}. The large size and high accuracy requirements of aerospace CFRP composite parts impose increased demands on deformation control during the curing process \cite{shen2022self}. Regulating the curing temperature distribution of a part is an effective means of controlling curing deformation. Therefore, constructing the predictive model of the temperature-to-deformation field on the geometry can provide essential support for further curing process optimisation \cite{struzziero2019numerical}. 

\begin{figure}[h]
\centering
\includegraphics[width=1\textwidth]{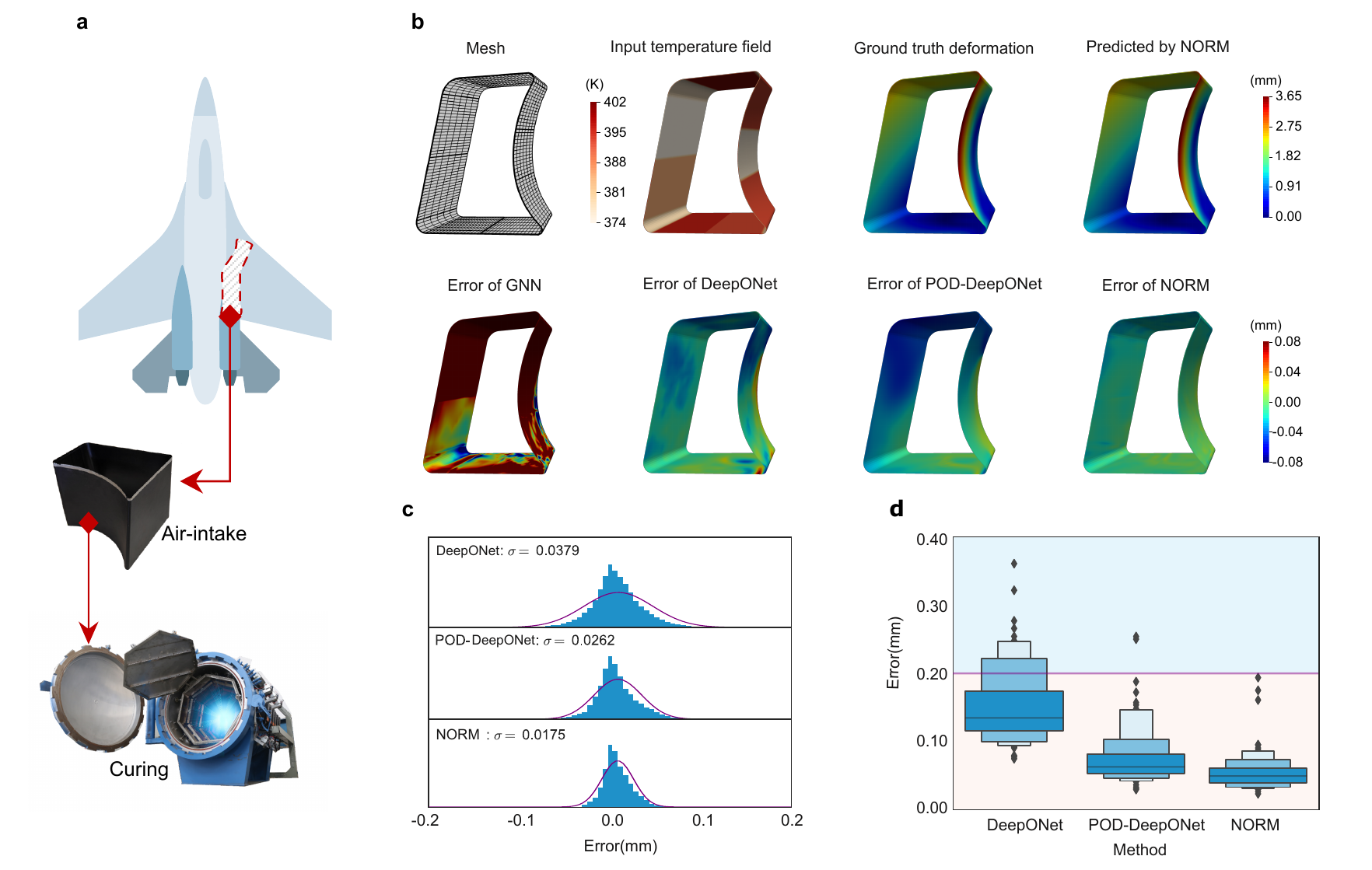}
\caption{\textbf{Composite workpiece deformation prediction case (Case 4).} \textbf{a}, Illustration of the air-intake workpiece and the composite curing. \textbf{b}, The input and output of the operator learning problem, the predicted deformation of NORM, and the prediction error of comparison methods. \textbf{c}, The distribution of deformation prediction error over all nodes of all test samples. \textbf{d}, The maximum prediction errors of all test cases for the three methods. }
\label{fig_composite_v2}
\end{figure}

As shown in Fig. \ref{fig_composite_v2}a, the CFRP composite workpiece used for the case study is the air-intake structural part of a jet. This workpiece is a complex closed revolving structure formed by multiple curved surfaces, which would deform significantly after high-temperature curing. The learning problem of this case is defined as the mapping from the temperature field $a(x,y,z)$ to the deformation field $u(x,y,z)$ on the given composite part. 

Fig. \ref{fig_composite_v2}b shows the prediction result of NORM and the prediction error of baseline methods of one test sample. It can be found that the error map of NORM is almost 'green' for the whole part, which means that predicted deformation field is very close to the reference value. Table. \ref{tab:Composite_result} shows that NORM outperforms all baseline methods in both $E_{L_2}$ and MME. Fig. \ref{fig_composite_v2}c shows the distribution of prediction error over all nodes of all test samples. It can be seen that the prediction errors of all nodes for all methods show Gaussian distributions with mean values approximating zero. The estimated standard deviations of different methods are marked in each figure. By comparison, NORM can achieve a lower prediction error uniformly and comprehensively for most nodes.

Composite manufacturing is a risk-sensitive problem, so it is not sufficient to consider only the relative $L_2$ error and average statistical results. According to the deformation prediction evaluation criteria provided by the engineers of the collaborating company, the maximum prediction error of the deformation field predicted by the data-driven model should be less than 0.2mm. Fig. \ref{fig_composite_v2}d reports the maximum prediction errors of all test cases. NORM is not only far outperforming the comparative methods but also has all test samples with a maximum prediction error of less than 0.2mm. 

\subsection{Blood flow dynamics prediction (Case 5)}

\begin{figure}[t]
\centering
\includegraphics[width=\textwidth]{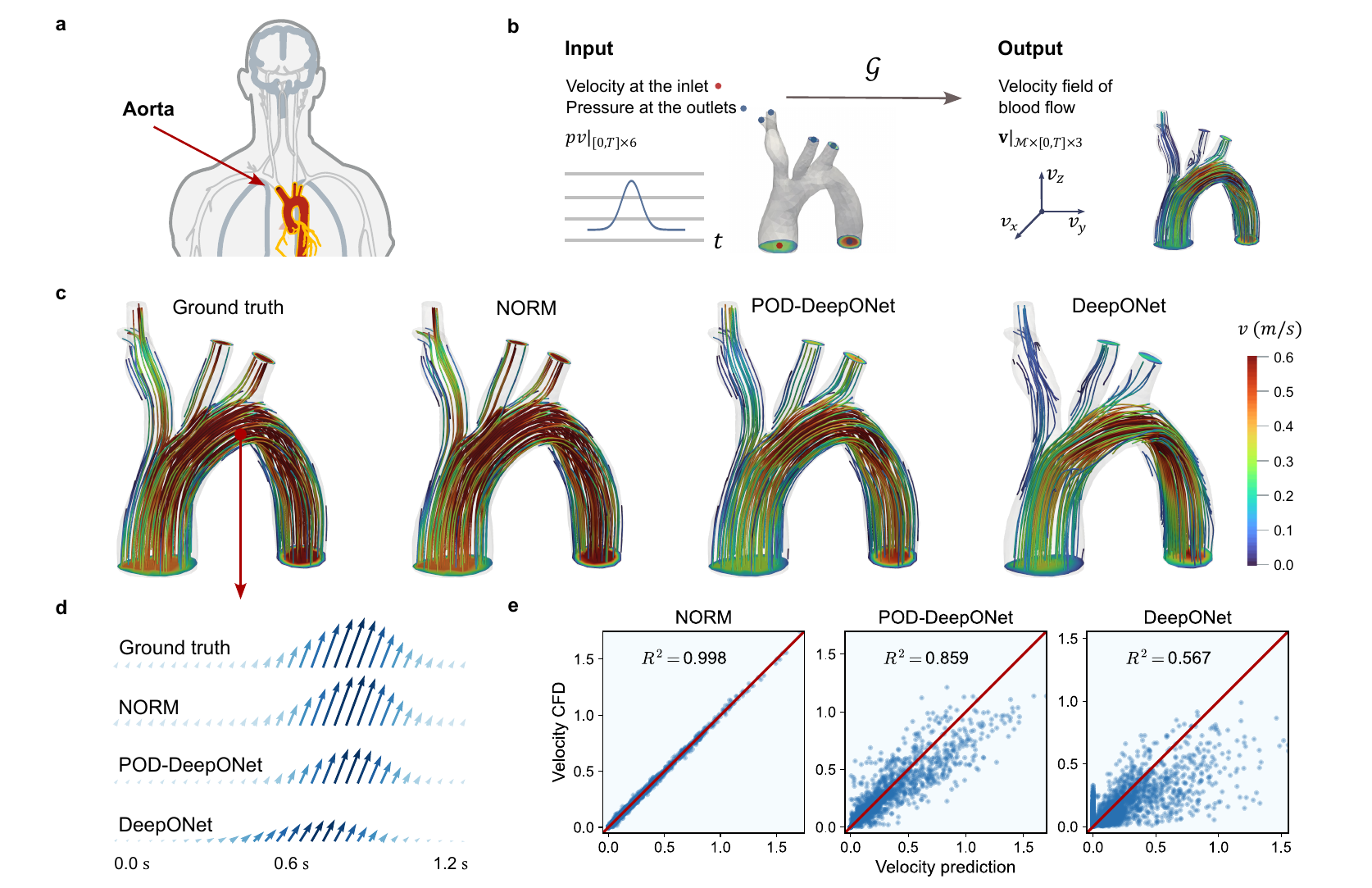}
\caption{\textbf{Blood flow dynamics prediction case (Case 5).} \textbf{a}, Illustration of the human thoracic aorta, the largest human artery. \textbf{b}, Illustration of the operator mapping $\cG$. The inputs are the velocity at the inlet and the pressure at the outlets. The output is the velocity field of the blood flow. \textbf{c}, Visualisation of the velocity streamlines (snapshots at a representative time) against baseline methods. \textbf{d}, Comparison of node velocity evolution prediction over time. We project the 3D vector onto the xy-plane. \textbf{e}, Comparison between ground truth and predictions for the magnitude of the velocity vector. We randomly sample 5000 spatiotemporal nodes from all test samples.}
\label{fig_blood_flow_result}
\end{figure}

Blood flow dynamics is the science of studying the characteristics and regularities of blood movement and its constituents in the organism, which is closely related to human health \cite{secomb2016hemodynamics}. To explore the potential of NORM for aortic hemodynamic modelling (Fig. \ref{fig_blood_flow_result}a), we consider a similar scenario as described in reference \cite{maul2023transient} where the inputs $\left.p v\right|_{[0, T] \times 6}$ are time-varying pressure and velocity at the inlet/outlets, and the output $\left.\mathbf{v}\right|_{\mathcal{M} \times[0, T] \times 3}$ is the velocity field of blood flow consisting of velocity components in three directions \cite{wen2010investigation}, as shown in Fig. \ref{fig_blood_flow_result}b. The spatial domain is represented by a tetrahedral mesh with 1656 nodes, and the temporal domain is discrete with 121 temporal nodes. It is worth pointing out that the challenges of this case lie in two aspects: 1) time-space complexity, i.e. the output function defined on the complex geometric domain is time-varying; 2) unbalanced node values, i.e. the velocity of most nodes is close to zero due to the no-slip boundary condition.

Since the Fourier basis is also a group of the LBO eigenfunction, NORM can naturally deal with the temporal dimension of input and output functions using the Fourier basis, as discussed in Supplementary Materials S1.2. Hence NORM adopted the structure of Fig. S1c. Statistics results of the NORM and two benchmarks (DeepONet and POD-DeepONet) are presented in Table \ref{tab2}. It is evident that NORM yields the smallest MME and relative $L_2$ error with minor variation. It stands to reason that at nodes $v \rightarrow 0$, even a slight prediction bias would lead to a significant relative $L_2$ error, but the proposed NORM achieves an impressive relative $L_2$ error $4.822 \%$, compared with 89.26\% of DeepONet and 37.42\% of POD-DeepONet, which demonstrates the remarkable approximation capability of NORM. Fig. \ref{fig_blood_flow_result}c compares the visualisation of the velocity streamlines (snapshots at a representative time) against baseline methods. We observe that NORM achieves an excellent agreement with the corresponding ground truth, while POD-DeepONet and DeepONet only learn the general trend of velocity distribution but lose the predictive accuracy of the node value. Especially, DeepONet fails to capture the local details of streamlines at inlets and outlets. Additional comparison visualisations of other moments can be found in Supplementary Materials S6.

Furthermore, Fig. \ref{fig_blood_flow_result}d provides the perspective to investigate the predictive accuracy of the node velocity evolution over time, which projects the three-dimensional vector onto the xy-plane. NORM agrees well with ground truth regarding phase and amplitude, while POD-DeepONet shows a smaller overall amplitude, and DeepONet loses accuracy in both aspects. Finally, the comparison between ground truth and predictions for the magnitude of the velocity vector at 5000 spatiotemporal nodes randomly sampled from all test samples is plotted in Fig. \ref{fig_blood_flow_result}e. Compared to NORM ($R^2=0.998$), despite a quasi-linear relationship maintained by POD-DeepONet ($R^2=0.859$), a prediction bias amplifies as the velocity increases. We conjecture that it is the approximation bias introduced by using the linear superposition method to fit complex nonlinear problems. As for DeepONet ($R^2=0.567$), since its training mode is point-wise and the loss function used for training is the relative $L_2$ error, the updating of the model parameters is mainly driven by the nodes $v \rightarrow 0$. Then, the model outputs tend to be zero, resulting in a trade-off with the optimisation of other nodes. Therefore, the overall prediction of nodes in DeepONet appears more dispersed and does not show a linear relationship.

\subsection{Analysis}

The encoder and the decoder of NORM are constructed by the spectral decomposition and the spectral reconstruction on the corresponding LBO eigenfunctions. This prompts a natural question: Could there be a more suitable basis than LBO eigenfunctions? From a model reduction point of view, the Proper Orthogonal Decomposition (POD) could also provide the projection basis to construct the encoder and the decoder. Consequently, NORM could be naturally extended to POD-NORM, wherein the POD modes of the training dataset replace the LBO eigenfunctions. Note that, the input data and the output data have different POD modes, so the structure of POD-NORM is similar to NORM with different input and output manifolds (Fig. S1b in the Supplementary Materials). Therefore, NORM and POD-NORM were compared to demonstrate the advantages of LBO eigenfunctions. We focus on the performance comparison of the Darcy problem and the composite case, because the input fields of these two tasks are more complex, which brings more challenge to the representability of the spectrum. The data results reported in Fig. \ref{fig_analysis} are the average based on five repeated runs.

\begin{figure}[t]
\centering
\includegraphics[width=\textwidth]{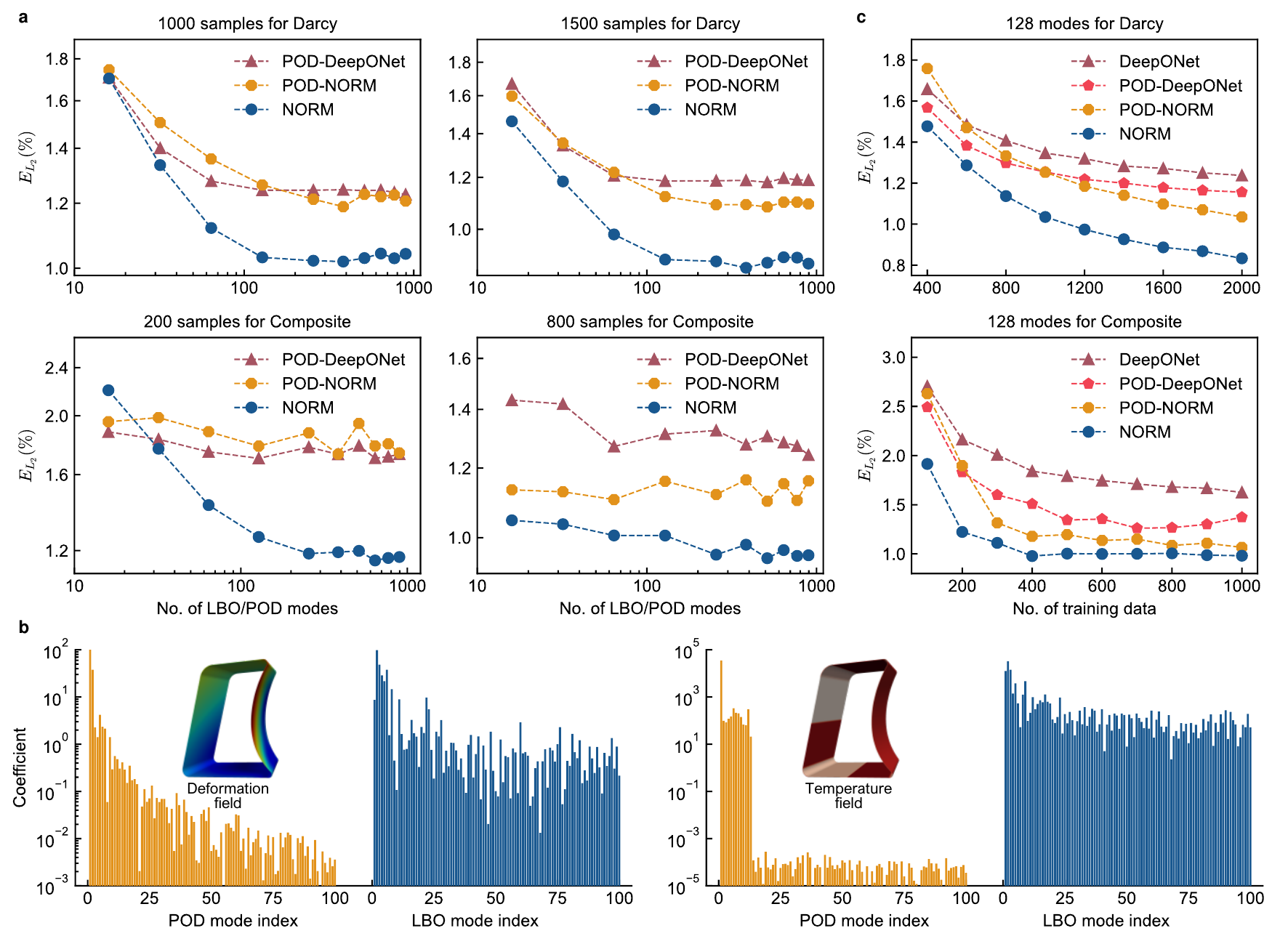}
\caption{\textbf{Analysis results for different methods in the Darcy case and the composite case.} \textbf{a}, Comparison of POD-DeepONet, POD-NORM, and NORM for various numbers of modes in different sizes of the training dataset. \textbf{b}, The coefficient analysis of the spectral decomposition of both the input temperature field and the output deformation field for the composite case using LBO and POD modes. \textbf{c}, Comparison of DeepONet, POD-DeepONet, POD-NORM, and NORM for different size of training data while the number of LBO/POD modes is 128.}
\label{fig_analysis}
\end{figure}

We first compared the performance of NORM, POD-NORM, and POD-DeepONet across various numbers of modes from 16 to 896. Fig. \ref{fig_analysis}a shows the error tendency of the different methods with different mode numbers. Each case contains the results with two different sizes of the training dataset, $\{1000, 1500\}$ for the Darcy case and $\{200, 800\}$ for the composite case. For the Darcy case, the prediction errors of all three methods decrease rapidly as the number of modes increases, eventually converging to a stable performance level. Notably, POD-NORM and POD-DeepONet have similar performance, and NORM shows smaller errors under all number of modes. These findings indicate that the LBO eigenfunctions possess a more robust representation capability compared to POD modes. In Fig. \ref{fig_analysis}a, we can observe that, in the composite case, increasing the number of POD modes does not appear to reduce the prediction error of POD-DeepONet and POD-NORM significantly. In contrast, NORM continues to show a clear decreasing trend in error while maintaining its leading performance.

To further explain the performance difference between the two modes in the composite case, we conducted a comparative analysis of the spectral decomposition of both the input temperature field and the output deformation field using LBO and POD modes. As shown in Fig. \ref{fig_analysis}b, the top 100 POD decomposition coefficients of the deformation field decreases rapidly from magnitudes of $10^2$ to $10^{-3}$, and the decomposition coefficients of the temperature field drop from $10^1$ to $10^{-4}$ suddenly. That indicates that the feature representation after the encoder $\cE$ contains coefficients spanning a wide range, from $10^{-4}$ to $10^2$, which could bring challenges for the learning process. Besides, since the high-order POD coefficients of the deformation field are extremely small, any errors in these coefficients could lead to significant sensitivity in the reconstructed results generated by the decoder $\cD$. By comparison, the LBO decomposition coefficients fluctuate in a relatively smaller range. This observation provides a potential explanation for why NORM consistently outperforms POD-NORM in most scenarios.

Another key distinction between these two modes lies in their underlying principles. POD modes are learnt from data, making their accuracy and generalisability heavily dependent on the size of training data. In contrast, LBO eigenfunctions are entirely independent of the training data. Therefore, we further compare the error tendency for different operator learning methods with respect to the training dataset size. For the Darcy problem, the training dataset sizes vary from 400 to 2000, and the test dataset is an additional 200 groups labelled data. For the composite case, the training dataset sizes are set from 100 to 1000, and another 100 groups are defined as the test dataset. The number of modes is consistently set to 128 for POD-DeepONet, POD-NORM, and NORM. The results are presented in Fig. \ref{fig_analysis}c. Notably, we observe that NORM exhibits a more rapid convergence rate as the training dataset increases, outperforming the other methods. In particular, for the Darcy problem, a NORM with 1200 samples can achieve a relative L2 error of less than 1.00\%, while DeepONet, POD-DeepONet, and POD-NORM with 2000 samples are 1.04\%, 1.24\% and 1.16\% respectively. To sum up, integrating LBO eigenfunctions enables NORM with superior performance bound and enhances the convergence capability. 

\section{Discussion}\label{sec:Discussion}

In this research, we propose a deep learning framework with a new concept called the Neural Operator on Riemannian Manifolds (NORM), to learn mappings between functions defined on complex geometries, which are common challenges in science discovery and engineering applications. Unlike existing neural operator methods (such as FNO, UNO, WNO) that rely on regular geometric domains of Euclidean structure, NORM is able to learn mappings between input and output functions defined on any Riemannian manifolds via LBO subspace approximation. Furthermore, the optimality of LBO eigenfunctions allows NORM to capture the global feature of complex geometries with only a limited number of modes, rather than directly learning the operator in the high-dimensional coordinate space. The ability of LBO eigenfunctions to approximate functions on Riemannian manifolds also guarantees the universal approximation property of NORM. 

NORM generalises the neural operator from Euclidean spaces to Riemannian manifolds, which has a wide range of potential applications, including PDEs solving, aerodynamics optimisation and other complex modelling scenarios. The case studies in parametric PDEs solving problems and engineering applications demonstrated that NORM can learn operators accurately and outperform the baseline methods. The discretisation-independent ability enables NORM with greater performance advantages compared with the coordinate spaces based model (such as DeepONet \cite{lu2021learning}) when learning more complex operators (Blood flow dynamics case). The architecture of NORM draws inspiration from the iterative kernel integration structure employed in FNO \cite{li2020fourier}. Notably, since the Fourier basis is also a group of the LBO eigenfunction, NORM can be treated as a generalisation of FNO from the Euclidean space to Riemannian manifolds. In addition, NORM can deal with different input/output manifolds, including Euclidean space or complex geometries, and thus has broader application potential compared with GNN or FNO, which requires the input and output to be the same domains. 

Although NORM shows promising performance in learning operators on complex geometries, the integration of LBO eigenfunctions also restricts the geometries to be Riemannian manifolds, which means that NORM could not deal with non-Riemannian geometries or even non-manifold geometries. For non-Riemannian geometries such as 3D point clouds, one feasible solution could be manually constructing the Riemannian metric $g$ from the point cloud and then calculating LBO eigenfunctions like described in reference \cite{yan2023spectral}. Recent researchers have also started to develop Laplacian for non-manifold triangle meshes, which could be a potential solution for operator learning in non-manifold geometries \cite{sharp2020laplacian}. 

Our method offers a new perspective for learning operators and solving PDEs on manifolds. Furthermore, the Laplacian-based approximation block in our method has strong extension potential to other neural operator structures or even physics-informed machine learning methods. For instance, the approximation block $\cN$ could replace the branch net of DeepONet, making the new framework discretisation independent in both input and output functions. When solving PDEs with known equations, integrating the approximation block $\cN$ into the physics-informed neural network could reduce the parameterisation complexity in coordinate spaces. In addition, the advantages of LBO eigenfunctions could be further discovered for more operator learning settings.

\section*{Acknowledgements}
This work was supported by the National Science Fund for Distinguished Young Scholars (No. 51925505), the General Program of National Natural Science Foundation of China (No. 52275491), the Major Program of the National Natural Science Foundation of China (No. 52090052), Joint Funds of the National Natural Science Foundation of China (No. U21B2081), the National Key R\&D Program of China (No. 2022YFB3402600), and New Cornerstone Science Foundation through the XPLORER PRIZE. 


\newpage

\FloatBarrier

\bibliographystyle{unsrt}

\bibliography{main}

\setcounter{figure}{0}
\renewcommand{\figurename}{Extended Figure}

\clearpage
\textbf{Extended Figure 1}
\begin{figure}[h]
\centering
\includegraphics[width=\textwidth]{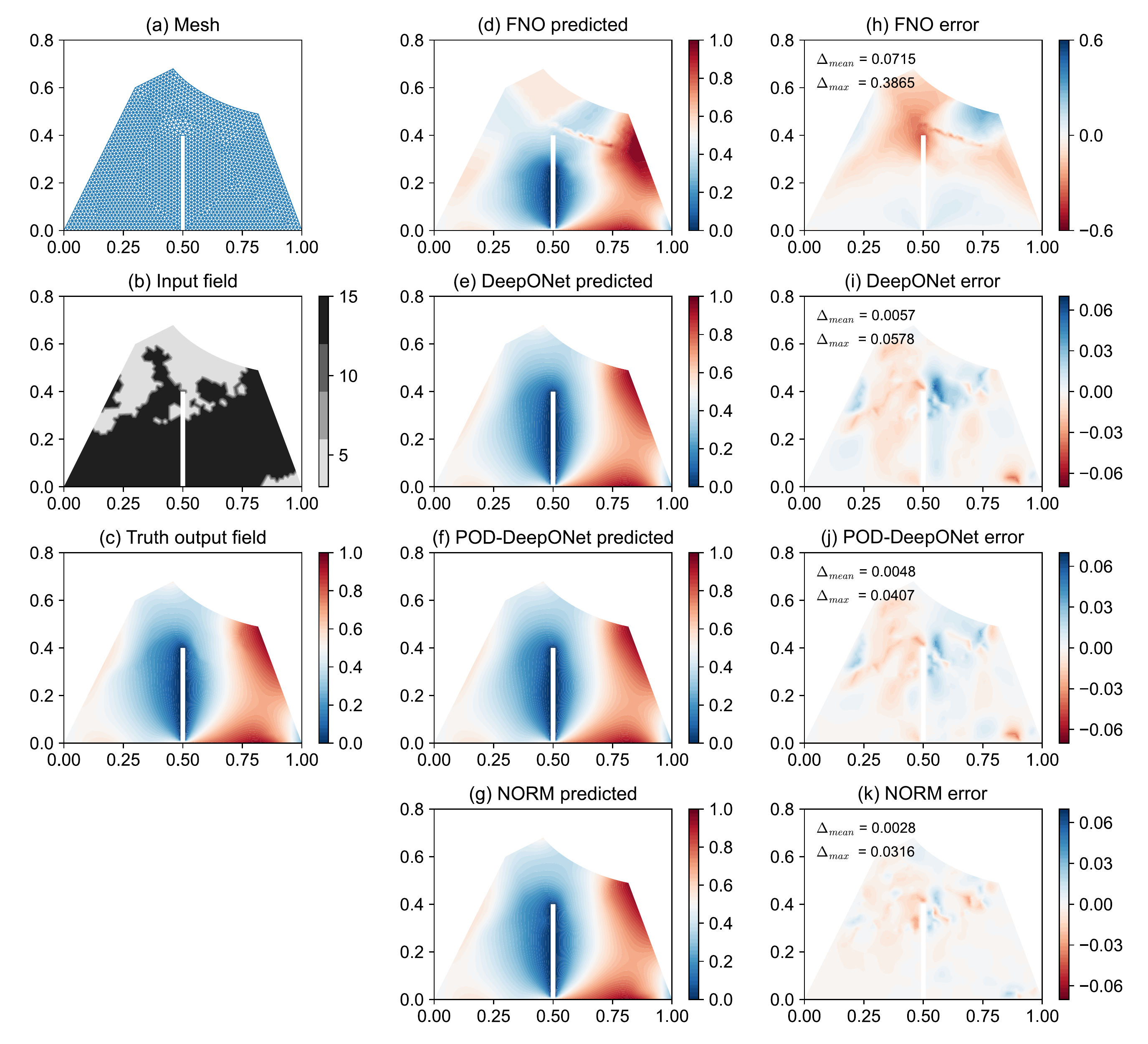}
\caption{\textbf{Experimental results of the Darcy problem.} \textbf{a}, The mesh for the irregular geometric domain. \textbf{b}, \textbf{c}, The input and output fields for a representative sample. \textbf{d-g}, The prediction results of different methods. \textbf{h-k}, The prediction errors of different methods.}
\label{darcy_Node2k_2}
\end{figure}

\clearpage
\textbf{Extended Figure 2}
\begin{figure}[h]
\centering
\includegraphics[width=\textwidth]{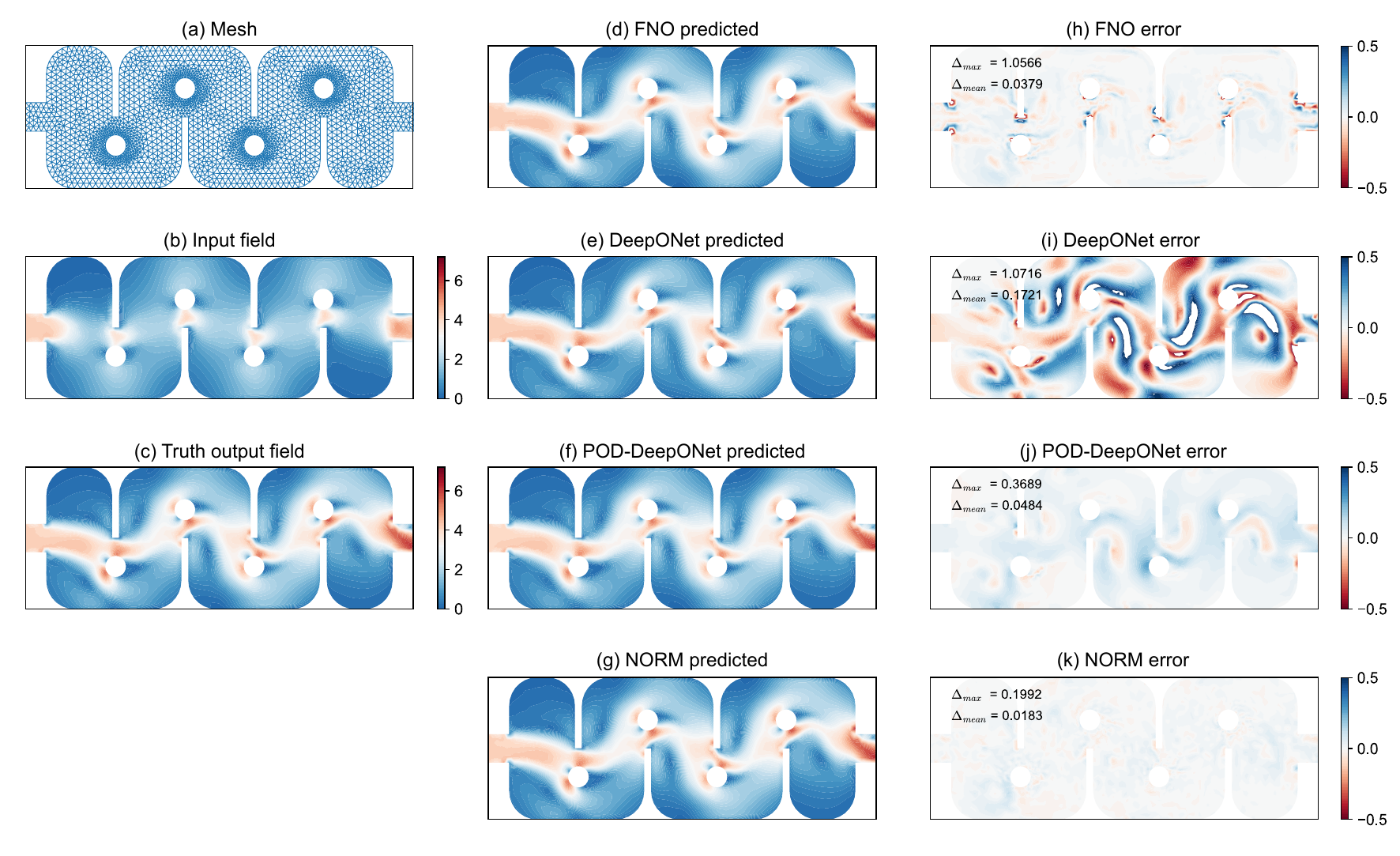}
\caption{\textbf{Experimental results of the pipe turbulence.} \textbf{a}, The mesh for the complex pipe shape. \textbf{b}, \textbf{c}, The input and output fields for a representative sample. \textbf{d-g}, The prediction results of different methods. \textbf{h-k}, The prediction errors of different methods.}
\label{pipe_flow_50_temp}
\end{figure}

\clearpage
\textbf{Extended Figure 3}
\begin{figure}[h]
\centering
\includegraphics[width=\textwidth]{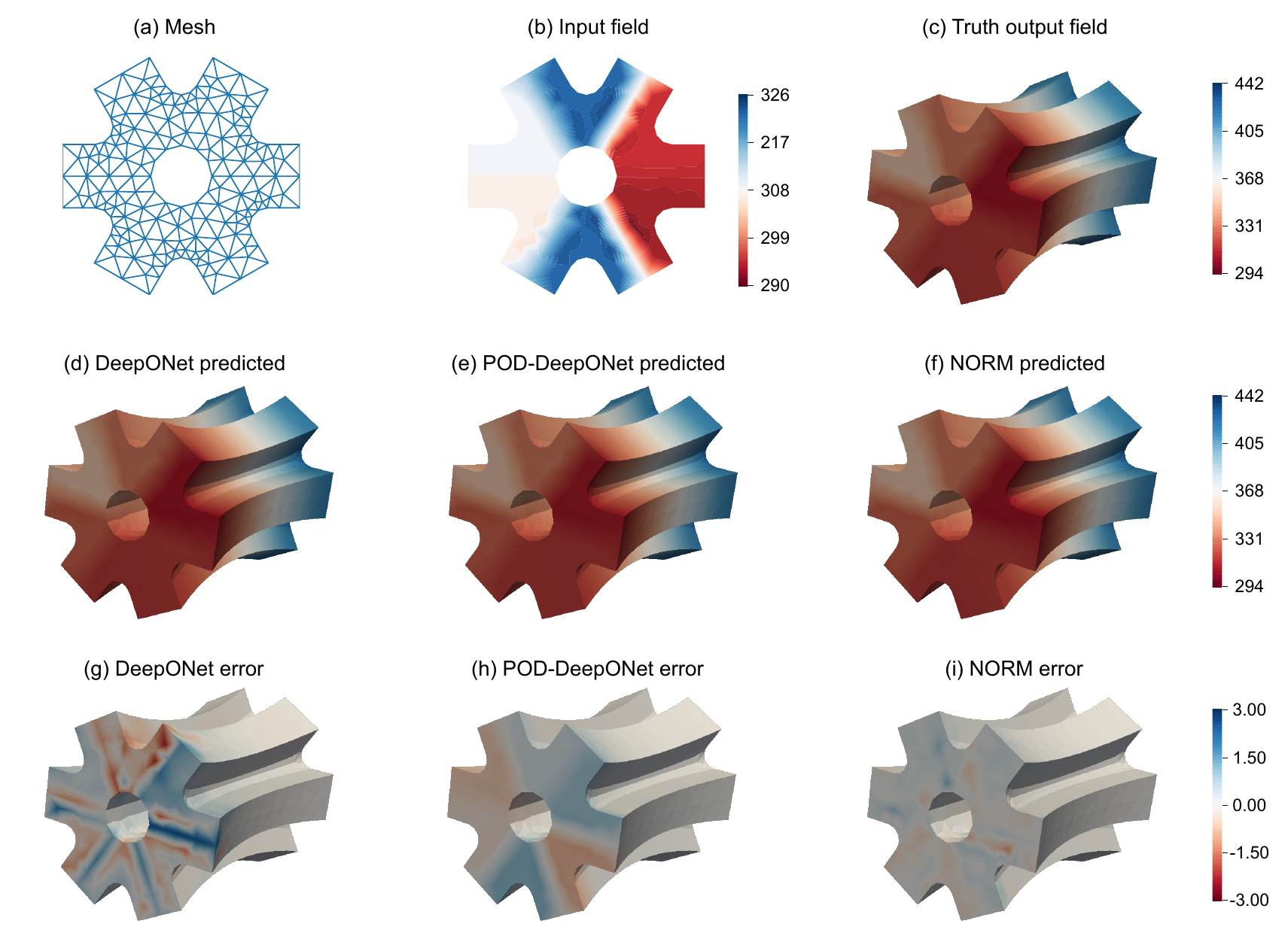}
\caption{\textbf{Experimental results of the heat transfer.} \textbf{a}, The mesh for the input geometric domain. \textbf{b}, \textbf{c}, The input and output fields for a representative sample. \textbf{d-f}, The prediction results of different methods. \textbf{g-i}, The prediction errors of different methods.}
\label{heat_transfer_all}
\end{figure}

\end{document}


\title{\textbf{Supplementary Materials for Learning Neural Operators on Riemannian Manifolds}}

\author[a,1]{Gengxiang Chen}
\author[b,1]{{Xu Liu}}
\author[a]{Qinglu Meng}
\author[a]{Lu Chen}
\author[a]{Changqing Liu}
\author[a]{{Yingguang Li}\corref{cor1}}
\address[a]{College of Mechanical and Electrical Engineering, Nanjing University of Aeronautics and Astronautics, 210016, Nanjing, China}

\address[b]{School of Mechanical and Power Engineering, Nanjing Tech University, 211816, Nanjing, China}

\cortext[cor1]{Corresponding author: liyingguang@nuaa.edu.cn}
\fntext[fn1]{Gengxiang Chen and Xu Liu contributed equally.}

\maketitle

\renewcommand\thefigure{S\arabic{figure}}
\renewcommand\thetable{S\arabic{table}}  
\renewcommand\thesection{S\arabic{section}} 

\tableofcontents

\linespread{1.3}


\clearpage
\section{Methodology}

\subsection{Discretised version of $\cN(v)$}

We first start from a common situation, assuming that the input and output functions are defined on the same manifold $\cM$, i.e. $\cX = \cY = \cM$. Suppose the manifold $\cM$ is discretised into $n_x$ nodes. For the approximation block  $\cN(v_l) =  \cD \circ \cR \circ \cE (v_l)$ in the NORM hidden layer, the input function $v_l(x)$ can be represented discretely as $\textbf{V}_l \in \mathbb{R}^{n_x \times d_v}$. The LBO eigenfunctions $\phi_{\cM,i}$  is then discretised into a vector form $\phi_{\cM,i} \in \mathbb{R}^{n_x\times1}$. Suppose we consider $d_m$ modes for the spectral decomposition and reconstruction, then all $d_m$ LBO eigenfunctions form a matrix $\Phi \in \mathbb{R}^{n_x \times d_m }$. The complex geometric information of the domain has been embedded in the LBO eigenfunctions.

First, the encoder of the discretised input matrix $\textbf{V}_l$ can be expressed as: 
\begin{equation}\label{eq: step1}
 \mathcal{E}(\textbf{V}_l)=\Phi^{\dagger}\textbf{V}_l
\end{equation}
where $\Phi^{\dagger}\in \mathbb{R}^{ d_m \times n_x}$ refers to the pseudo inverse of the LBO eigenfunctions matrix $\Phi$, defined as: 
\begin{equation}
\Phi^{\dagger} = (\Phi^{\top}\Phi)^{-1}\Phi^{\top}
\end{equation}

Denote $\mathcal{R}$ as a simple linear mapping $\textbf{R} \in \mathbb{R}^{ d_m \times d_v \times d_v}$, the mappings on the encoded information can then be represented as: 
\begin{equation}\label{eq: step1}
\mathcal{R} \circ \mathcal{E}(\textbf{V}_l)=\left(\textbf{R}\cdot\left(\Phi^{\dagger}\textbf{V}_l\right)\right)
\end{equation}
where the tensor operation is defined as:  
\begin{equation}\label{tensor_product}
\left(\textbf{R}\cdot\left(\Phi^{\dagger}\textbf{V}_l\right)\right)_{k, l}=\sum_{j=1}^{d_v} \textbf{R}_{k, l, j}(\Phi^{\dagger}\textbf{V}_l)_{k, j}, \quad k=1, \ldots, d_m, \quad l=1, \ldots, d_v
\end{equation}

The decoder process is simply the linear transformation with the LBO eigenfunctions matrix. Then the discretised version of $\cN(v)$ can be given as:

\begin{equation}\label{laplace_layer}
\mathcal{D} \circ \mathcal{R} \circ \mathcal{E}= \Phi\left(\textbf{R}\cdot\left(\Phi^{\dagger}\textbf{V}_l\right)\right) 
\end{equation}

\subsection{Model structure with different input and output manifolds}

NORM can deal with different input and output manifolds by defining different L-layers. Different manifolds mean that the input and output have different LBO eigenfunctions, which will influence the encoder $\cE$ and decoder $\cD$ of the approximation block $\cN$ in each L-layer. Fig. \ref{fig:3manifolds}a shows the common situation where the input and output functions are both defined on the same manifolds $\cM$, which means all $\cE$ and $\cD$ in all L-layers take the same LBO eigenfunctions (marked as $\cM \rightarrow \cM$ for illustration).

As shown in Fig. \ref{fig:3manifolds}b, when the input and output functions are defined on different manifolds, there will be three different types of L-layers. The beginning L-layers ($\cX \rightarrow \cX$) employ LBO eigenfunctions of manifold $\cX$ for both the encoder $\cE$ and decoder $\cD$. The middle L-layer should utilise the LBO eigenfunctions of manifold $\cX$ for the encoder $\cE$, and the LBO eigenfunctions of manifold $\cY$ for the decoder $\cD$. Therefore, the output of the middle L-layer ($\cX \rightarrow \cY$) will have the same domain discretisation dimension of $\cY$. After that, the feature can be passed to the ending L-layers ($\cY \rightarrow \cY$) with the decoder and encoder defined with LBO eigenfunctions of manifold $\cY$.

Since the Fourier basis can also be regarded as a group of the LBO eigenfunction, NORM can be treated as a generalisation of FNO from the Euclidean space to Riemannian manifolds. Therefore, NORM can also deal with temporal functions as input or output. Fig. \ref{fig:3manifolds}c shows the problem with input temporal function and output spatial function. Similarly, the beginning multiple L-layers can define the decoder and encoder with the Fourier basis to process the input temporal function ($\cF \rightarrow \cF$). The middle L-layer can use the Fourier basis for the encoder $\cE$, and take the LBO eigenfunctions of manifold $\cY$ for the decoder $\cD$ (marked as $\cF \rightarrow \cY$). 

\begin{figure}[h]
\centering
\includegraphics[width=\textwidth]{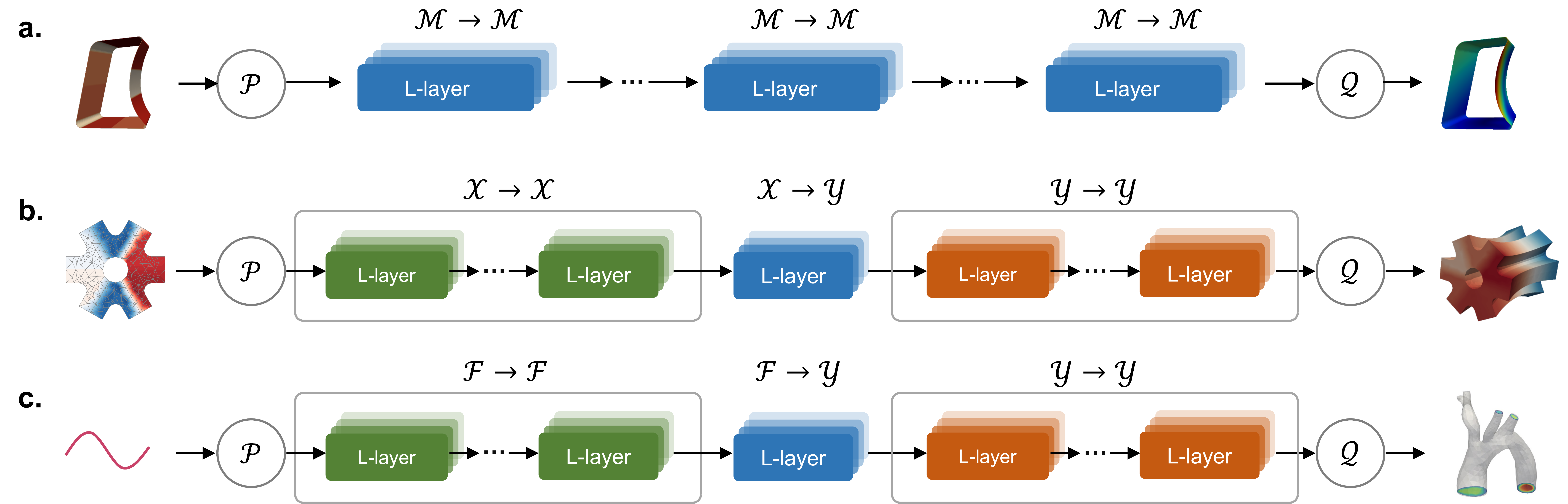}
\caption{Model structure with different input and output manifolds}
\label{fig:3manifolds}
\end{figure}

\clearpage

\section{About LBO eigenfunctions}

\subsection{Discretised LBO of triangular mesh}
The Laplace-Beltrami Operators (LBOs) of different geometric meshes are strictly defined in the differential geometry field \cite{alexa2020properties}, including triangular mesh, quadrilateral mesh, tetrahedral mesh, etc. Take the example of the triangular mesh, as shown in Fig. \ref{fig: contan_laplace} and Eq. \ref{eq: contan_laplace}, the discrete Laplacian of a scalar function $f$ on a vertex $i$ is defined by the cotangent function of the adjacent nodes, where $\mathcal{M}({i})$ is the vertex $i$ on the geometric mesh. The Laplacian of triangular mesh is also called the cotangent Laplace operator, which can be derived in many different ways, including finite analysis, finite volume method, or discrete exterior calculus \cite{crane2018discrete}.

\begin{equation}\label{eq: contan_laplace}
(\Delta f)_{{ij}} \approx \frac{1}{2} \sum_{{j} \in \mathcal{M}({i})}\left(\cot \alpha_{{ij}}+\cot \beta_{{ij}}\right)\left(f_{{i}}-f_{{j}}\right)
\end{equation}

After defining the LBO of complex geometries, the LBO eigenfunctions can be obtained by solving the eigenfunctions $\Delta \phi =\lambda \phi$ with Galerkin method, power iteration, or other numerical methods \cite{vallet2008spectral,neumann2014compressed}. 

\begin{figure}[h]
\centering
\includegraphics[width=5cm]{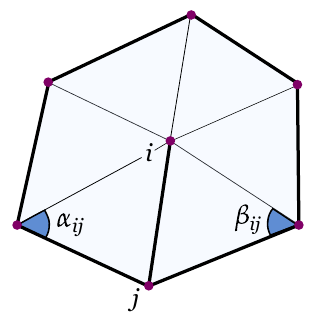}
\caption{Cotangent Laplace operator of the triangular mesh.}
\label{fig: contan_laplace}
\end{figure}

\subsection{Projection error of LBO eigenfunctions}

\begin{Theo}[\textbf{Projection error of LBO eigenfunctions}] \cite{aflalo2013spectral}  
\label{theorem_lbo_error} 
Let  $\cM$ be a given Riemannian manifold, an induced LBO $\Delta$, with associated spectral basis $\phi_i$, where $\Delta \phi_i=\lambda_i \phi_i$. Consider a smooth function $f \in L^2(\cM , \R)$,  the projection error is:
\begin{equation}
\left\|r_n\right\|^2_{L^2} \equiv\left\|f-\sum_{i=1}^n\left\langle f, \phi_i\right\rangle \phi_i\right\|^2_{L^2} \leq \frac{\left\|\nabla f\right\|^2_{L^2}}{\lambda_{n+1}}
\end{equation}

and $\left\|r_n\right\|_{L^2} \rightarrow 0$ as $n \rightarrow \infty$.

\end{Theo}

\begin{proof}[Proof of Theorem \ref{theorem_lbo_error}] \cite{aflalo2013spectral}  :

Define the projection residual function as $r_n = f-\sum_{i=1}^n\left\langle f, \phi_i\right\rangle \phi_i$. Since $\left\langle  \phi_i, 
 \phi_j\right\rangle = \delta_{ij}$, it is easy to verify that $\left\langle  r_n, 
 \phi_i\right\rangle = 0, \forall i, 1\leq i \leq n$. The projection error can be given as:
\begin{equation}
\left\|r_n\right\|^2_{L^2} = \left\|\sum_{i=n+1}^{\infty}\left\langle r_n, \phi_i\right\rangle \phi_i\right\|^2_{L^2} \\ = \sum_{i=n+1}^{\infty}\left\langle r_n, \phi_i\right\rangle^2 
\end{equation}

The gradient bound of residual function is:

\begin{equation}
\begin{aligned}
\left\|\nabla r_n\right\|^2_{L^2} & =\left\|\sum_{i=n+1}^{\infty}\left\langle r_n, \phi_i\right\rangle \nabla \phi_i\right\|^2_{L^2} =   \sum_{i=n+1}^{\infty} \lambda_{i}\left\langle r_n, \phi_i\right\rangle^2\\
\end{aligned}
\end{equation}

Since the eigenvalues are ascending, $\lambda_1 \leq \lambda_2, \ldots $, we have:

\begin{equation}
\begin{aligned}
\left\|\nabla r_n\right\|^2_{L^2} \geq \lambda_{n+1} \sum_{i=n+1}^{\infty}\left\langle r_n, \phi_i\right\rangle^2 = \lambda_{n+1} \left\|r_n\right\|^2_{L^2} 
\end{aligned}
\end{equation} 

So that:
\begin{equation}
\left\|r_n\right\|^2_{L^2}  \leq  \frac{\left\|\nabla r_n\right\|^2_{L^2}}{ \lambda_{n+1}}  
\end{equation}

Then, we build the connection between $\left\|\nabla r_n\right\|^2_{L^2}$ and $\left\|\nabla f\right\|^2_{L^2}$.  We have:

\begin{equation}
\left\|\nabla f\right\|^2_{L^2} = \left\|\nabla r_n + \sum_{i=1}^{n}\left\langle f, \phi_i\right\rangle \nabla \phi_i \right\|^2_{L^2}  = \left\|\nabla r_n\right\|^2_{L^2} +  \sum_{i=i}^{n} \lambda_{i}\left\langle f, \phi_i\right\rangle^2
\end{equation}

Then, it follows that:

\begin{equation}
\left\|r_n\right\|^2_{L^2}  \leq  \frac{\left\|\nabla r_n\right\|^2_{L^2}}{ \lambda_{n+1}}   \leq  \frac{\left\|\nabla f\right\|^2_{L^2}}{ \lambda_{n+1}}  
\end{equation}

By assumption, $f \in L^2$. Therefore, $\left\|\nabla f\right\|^2_{L^2}$ is bounded. Theorem 2.6 in Ref \cite{nica2011eigenvalues} shows that $\lim _{n \rightarrow \infty} \lambda_n=\infty$, so it can be verified that $\lambda_{n+1} \rightarrow \infty $ forces $\left\|r_n\right\|_{L^2}  \rightarrow 0$.

\end{proof}

\begin{Theo}\label{theorem_lbo_optimal} Optimality of LBO eigenfunctions \cite{aflalo2015optimality}  
Given a Riemannian manifold $\cM$, the induced LBO $\Delta$, and its spectral basis $\phi_i$, where $\Delta \phi_i=\lambda_i \phi_i$, and a real scalar value $0 \leq \alpha<1$. For any $f \in L^2(\cM , \R)$, there is no orthonormal basis of functions $\left\{\psi_i\right\}_{i=1}^{\infty}$, and an integer $n$ such that
\begin{equation}
\left\|f-\sum_{i=1}^n\left\langle f, \psi_i\right\rangle \psi_i\right\|^2_{L^2} \leq \alpha \frac{\left\|\nabla f\right\|^2_{L^2}}{\lambda_{n+1}}, \quad \forall f
\end{equation}
\end{Theo}

\clearpage

\section{About universal approximation theory}

\subsection{Non-Euclidean universal approximation condition}

Recently, Kratsios et al. \cite{kratsios2020non} investigated which modifications to the input and output of a neural network could deal with non-Euclidean while preserving the universal approximation capability. Based on their research, the Non-Euclidean Universal Approximation Condition can be summarised as follows:

\begin{Theo} [\textbf{Non-Euclidean Universal Approximation Condition} \cite{kratsios2020non}]

Let $\phi: \cX \rightarrow \R^m$ and $\rho: \R^n \rightarrow \cY$,  where  $\cX$ and  $\cY$ are topological spaces. Equip $C(\cX, \cY)$ with the bounded compact topology, $C\left(\R^m, \R^n\right)$ with the topology of uniform convergence on compacts, let $\cF$ be a subset of $C\left(\R^m, \R^n\right)$, and define the subset $\cF \rho, \phi \subseteq C(\cX, \cY)$ by
\begin{equation}
\cF_{\rho, \phi} \triangleq\{g \in C(\cX, \cY): g=\rho \circ f \circ \phi \text { , where } f \in \cF\}.
\end{equation}
Suppose that  $\cF$ is dense in $C\left(\R^m, \R^n\right)$. If $\phi$ is a continuous injective map, $\rho$ is a continuous surjective map, then  $\cF_{\rho, \phi}$ is dense in $C(\cX, \cY)$.
\label{non_eu}
\end{Theo} 

Here, $\cF_{\rho, \phi}$ is dense in $C(\cX, \cY)$ means that given any $\epsilon > 0$ and $g_C \in C(\cX, \cY)$, there exists $g \in \cF_{\rho, \phi} $ satisfying:
\begin{equation}
\sup _{a \in K}\|g - g_C\| \leq \epsilon
\end{equation}

For the defined neural operator on Riemannian manifolds $\cN= \cD \circ \cR \circ \cE$, suppose the approximator $\cR$ is a neural network that holds universal approximation property. Then, based on Theorem \ref{non_eu}, the key step of establishing an $\cN$ with universal approximation property is to construct a continuous injective map $\cE$ from functions on manifolds to the Euclidean space and a continuous surjective map $\cD$ from the Euclidean space to functions on manifolds. 

\subsection{Proof of universal approximation of NORM}

Let $\cN= \cD \circ \cR \circ \cE$ be a neural operator for $C(\cA, \cU)$, where $\cR$ represent a neural network that has universal approximation property in $C\left(\R^{\dx}, \R^{\dy}\right)$. The encoder is defined as: $ \cE:{\cA} \rightarrow\ \R^{\dx}$ and $\cE (a) := (\left\langle a, \phi_{\cX,1} \right\rangle, \ldots, \left\langle a, \phi_{\cX,\dx} \right\rangle), \forall a \in \cA$. The decoder is defined as $\cD:\R^{\dy}\rightarrow {\cU}$, and $\cD (\beta) =  \sum_{i=1}^{\dy} \beta_i \phi_{\cY,i} , \forall \, \beta \in \R^{\dy}$. $\cX$ and $\cY$ are Riemannian manifolds. $\cA$ and $\cU$ are $L^2$ spaces. $\phi_{\cX,i}$ and $\phi_{\cY,i}$ are LBO eigenfunctions of manifolds $\cX$ and $\cY$, respectively. The universal approximation theorem of neural operators on Riemannian manifolds is as follows:

\begin{Theo} [\textbf{Universal approximation theorem for the neural operator on Riemannian manifolds}]
Let $\cG: \cA(\cX;\R) \rightarrow \cU(\cY;\R)$ be a Lipschitz continuous operator, $K \in \cA$ is compact set. Then for any $\epsilon > 0$, there exists a neural operator on $\cN= \cD \circ \cR \circ \cE$, such that:
\begin{equation}
\sup _{a \in K}\|\mathcal{G}(a) -\mathcal{N}(a)\|_{L^2} \leq \epsilon
\end{equation}
\label{ua_norm}
\end{Theo} 
\begin{proof}[Proof of Theorem \ref{ua_norm}]:

Since $\cN $ is defined with a finite number of LBO eigenfunctions, the encoder $\cE$ is not injective, and the decoder $\cD$ is not surjective, so we cannot derive the universal approximation property directly based on Theorem \ref{non_eu}. 

The following proof consists of three steps: first, the universal approximation error on the projection subspace, then the decoding error on $\cY$, and the influence of the encoding error on $\cX$.

\textbf{Step 1: Approximation error on the projection subspace }

For the input function space $\cA$, let $V_{\cX, \dx} \subset \cA$ be the $\dx$-dimensional projection space of the LBO eigenfunctions of the manifold $\cX$, namely $V_{\cX, \dx}=\operatorname{span}\left\{\phi_{\cX,1}, \phi_{\cX,2}, \ldots, \phi_{\cX,\dx}\right\} \subset \cA$. Therefore, the orthogonal projection of the input function $a$ can be represented as $\Pi_{V_{\cX, \dx}}a $:
\begin{equation}
\Pi_{V_{\cX, \dx}}a =\sum_{i=1}^{\dx}\left\langle a, \phi_{\cX,i} \right\rangle \phi_{\cX,i}
\end{equation}

Similarly, for the output function space $\cU$, let $V_{\cY, \dy} \subset \cU$ be the $\dy$-dimensional projection space of the LBO eigenfunctions of the manifold $\cY$, namely $V_{\cY, \dy}=\operatorname{span}\left\{\phi_{\cY,1}, \phi_{\cY,2}, \ldots, \phi_{\cY,\dy}\right\} \subset \cU$. Then the orthogonal projection of the output function $u$ on $V_{\cY, \dy}$ can be defined as:
\begin{equation}
\Pi_{V_{\cY, \dy}}u =\sum_{i=1}^{\dy}\left\langle u, \phi_{\cY,i} \right\rangle \phi_{\cY,i}
\end{equation}

Let  $\cN^\dagger= \cD^\dagger \circ \cR \circ \cE^\dagger$ be a neural operator on the projection space $C(V_{\cX, \dx}, V_{\cY, \dy})$, where $\cR$ represent a neural network that has universal approximation property in $C\left(\R^{\dx}, \R^{\dy}\right)$. The encoder can be defined as the following mapping: 
\begin{equation}\label{eq:encoder}
\cE^\dagger:V_{\cX, \dx} \rightarrow\ \R^{\dx}, \quad 
\cE^\dagger (a) := (\left\langle a, \phi_{\cX,1} \right\rangle, \ldots, \left\langle a, \phi_{\cX,\dx} \right\rangle) 
\end{equation}

And the decoder can be given as: 
\begin{equation}\label{eq:decoder}
\cD^\dagger:\R^{\dy}\rightarrow V_{\cY, \dy}, \quad 
\cD^\dagger (\beta) =  \sum_{i=1}^{\dy} \beta_i \phi_{\cY,i} \quad \forall \, \beta \in \R^{\dy}
\end{equation}

Then $\cD^\dagger$ and $\cE^\dagger$ follow the assumption in Theorem \ref{non_eu}, that $\cE^\dagger$ is a continuous injective map and $\cD^\dagger$ is a continuous surjective map. Based on Theorem \ref{non_eu}, $\cN^\dagger$ is an universal approximator in $C(V_{\cX, \dx}, V_{\cY, \dy})$. Suppose $K$ is compact set in $\cA$, then for any $\epsilon > 0$, there exists a $\cN^\dagger$ and a $\dx  \in \mathbb{N} $, such that:
\begin{equation}
\sup _{a \in K}\|\cN^\dagger(\Pi_{V_{\cX, \dx}}a) -\Pi_{V_{\cY, \dy}}\cG(\Pi_{V_{\cX, \dx}}a)\|_{L^2} \leq \frac{\epsilon}{3}
\end{equation}

Note that, we have $\cE^\dagger(\Pi_{V_{\cX, \dx}}a) = \cE(a) , \forall a \in \cA$ and $\cD^\dagger(\beta) = \cD(\beta) , \quad \forall \, \beta \in \R^{\dy}$. And also $\cN^\dagger(\Pi_{V_{\cX, \dx}}a) = \cN(a) , \forall a \in \cA$. Hence, we have:
\begin{equation}\label{app_step1}
\sup _{a \in K}\|\cN(a)-\Pi_{V_{\cY, \dy}}\cG(\Pi_{V_{\cX, \dx}}a)\|_{L^2} \leq \frac{\epsilon}{3}
\end{equation}

\textbf{Step 2: Decoding error on the output}

Theorem \ref{theorem_lbo_error} shows that the projection error of LBO eigencfunctions convergence to 0 when with a sufficient number of basis. Therefore, for any $\epsilon > 0$, there exists a number $\dy \in \mathbb{N}$, such that:
\begin{equation}\label{app_step2}
\sup _{a \in K}\| \cG(\Pi_{V_{\cX, \dx}}a) - \Pi_{V_{\cY, \dy}}\cG(\Pi_{V_{\cX, \dx}}a) \|_{L^2} \leq \frac{\epsilon}{3}
\end{equation}

\textbf{Step 3: Encoding error on the input}  

Here we assume $\cG$ is Lipschitz continuous, that is, there exists a constant $M > 0$ that:
\begin{equation}\label{lip_cont}
\left\|\cG\left(a_1\right)-\cG\left(a_2\right)\right\|_{L^2} \leq M\left\|a_1-a_2\right\|_{L^2}, \quad \forall a_1, a_2 \in \cA
\end{equation}

Since $\Pi_{V_{\cX, \dx}}a$ can approximate $a$ at any accuracy, then for any $\epsilon > 0$, there exists $\dx \in \mathbb{N}$, such that:
\begin{equation}\label{app_step3}
\sup _{a \in K}\| \cG(\Pi_{V_{\cX, \dx}}a) - \cG(a) \|_{L^2} \leq \frac{\epsilon}{3}
\end{equation}

\textbf{Step 4: Combining the errors from steps 1 to 3}  

Therefore, triangle inequality implies that:
\begin{equation}\label{deal_4}
\begin{aligned}
\sup _{a \in {K}}\|\cN(a) -\cG(a) \|_{L^2} \leq & \|\cN(a)-\Pi_{V_{\cY, \dy}}\cG(\Pi_{V_{\cX, \dx}}a)\|_{L^2} + \\
&  \| \cG(\Pi_{V_{\cX, \dx}}a) - \Pi_{V_{\cY, \dy}}\cG(\Pi_{V_{\cX, \dx}}a) \|_{L^2} + \\
& \| \cG(\Pi_{V_{\cX, \dx}}a) - \cG(a) \|_{L^2} \\  \leq  & 
\epsilon
\end{aligned}
\end{equation}

This concludes the proof.
\end{proof}

According to the proof procedure above, one fundamental characteristic of NORM that supports its universal approximation property is the ability of LBO eigenfunctions to approximate continuous functions of Riemannian manifolds with arbitrary accuracy. Therefore, this proof procedure can be generalised to other potential extensions of NORM that utilise different orthogonal basis functions rather than LBO, as long as they can also approximate functions on Riemannian manifolds. 

\clearpage
\section{Data generation}
\subsection{Learning PDEs solution operators}
\subsubsection{Darcy problem (Case 1)}

Darcy flow equation is a classical law for describing the flow of a fluid through a porous medium. This problem is also widely used for various neural operator verification. We focus on the darcy equation on 2D irregular geometric domain, which can be described by the following equation:  

\begin{equation}
\label{eq_darcy}
-\nabla \cdot(a \nabla u )=f
\end{equation}
where $a$ is the diffusion coefficient field, $u$ is the pressure field and $f$ is the source term to be specified. The learning target in the Darcy flow problem is the mapping from the diffusion coefficient field $a(\mathbf{x})$ to the pressure field $u(\mathbf{x})$: 
\begin{equation}
\cG:  a(\mathbf{x}) \mapsto u(\mathbf{x}), \quad \mathbf{x} \in \cM
\end{equation}

For this case, the source term is set to 1, i.e. $f=1$. The input diffusion coefficient field $a(\mathbf{x})$ is generated by the Gaussian random field with a piecewise function, namely $a(\mathbf{x})= t(\mu)$, where $\mu$ is a distribution defined by $\mu=\cN\left(0,(-\Delta+25 I)^{-2}\right)$ \cite{kovachki2021neural}. After sampling from this distribution, the diffusion coefficient field $a(\mathbf{x})$ can be generated from the following piecewise function: 

\begin{equation}
\label{eq: input_generation}
t(\mu)= \begin{cases}12, & \mu \geq 0 \\ 4, & \mu<0\end{cases}
\end{equation}

We design an irregular geometric domain with a thin rectangle notch inside, which can increase the complexity of the learning problem. As shown in Fig. \ref{fig: darcy1}a, the geometric domain of the Darcy case is divided by the triangle mesh with 2290 nodes, where the outside boundary condition follows $u=u_{\partial D}(x)$ and the three boundaries of the inside rectangle follows $u=0$. The boundary condition $u_{\partial D}(x)$ is shown in Fig. \ref{fig: darcy1}b, and Fig. \ref{fig: darcy1}c-d show the input field and output field of one labelled data. In this case, 1200 labelled data are randomly generated, 1000 of them are used as the training data, and the rest 200 groups are defined as the test data.

\begin{figure}[h]
\centering
\includegraphics[width=16.5cm]{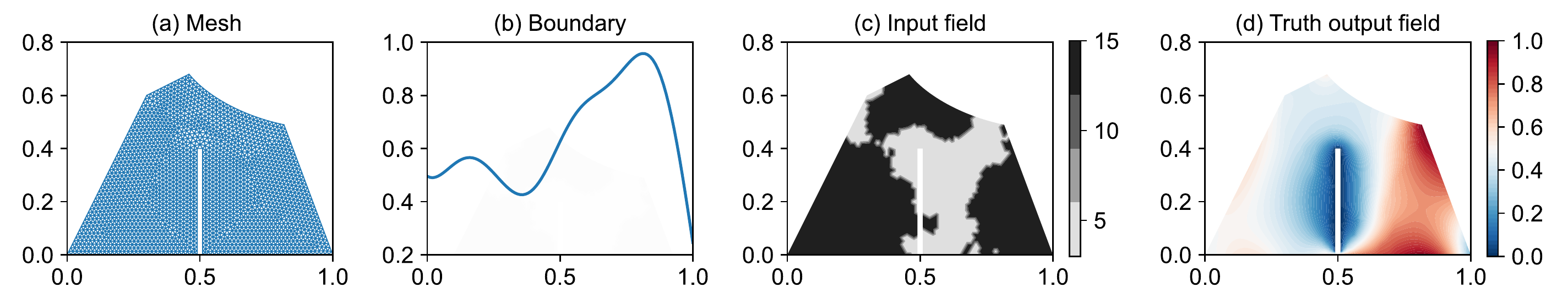}
\caption{The mesh, boundary conditions, input field and output field of the Darcy flow case 1.}
\label{fig: darcy1}
\end{figure}

\subsubsection{Pipe turbulence (Case 2)}

Flow in a pipe is very common in physiological systems, here we consider turbulent flows in a complex pipe, where the governing equation is the 2-d Navier-Stokes equation for a viscous, incompressible fluid:
\begin{equation}
\frac{\partial \boldsymbol{v}}{\partial t}+(\boldsymbol{v} \nabla) \boldsymbol{v}=-\nabla p+\mu \nabla^2 \boldsymbol{v}, \quad \nabla \cdot \boldsymbol{v}=0
\end{equation}
where $\boldsymbol{v}$ is the velocity, $p$ is the pressure, and the fluid chosen is water. We use the $k-\varepsilon$ model of Reynolds Average Navier-Stockes (RANS) models in the Comsol Multiphysics to conduct simulations. The geometry of the pipe is shown in Fig. \ref{fig: pipeflow}a. The average normal velocity $\boldsymbol{v}=[1,5]$ is imposed at the inlet, zero pressure condition is imposed at the outlet, and the no-slip boundary condition is imposed at the pipe surface. For given inlet velocity, we perform a $1\mathrm{~s}$ transient simulation to predict the velocity distribution in the pipe. 
The learning problem of this case is defined as the mapping from the velocity field of $t \in[0.1 s, 0.9 s]$ to the velocity field of $t+0.1 s$, as shown in Fig. \ref{fig: pipeflow}b-c. The input and output mesh both comprise 2673 nodes. Finally, we generate 400 trajectories, including 80 sets of transient simulations of different inlet velocities and 5 input-output pairs for each simulation. 300 of them are used as training data, and the rest of the 100 groups are defined as test data.

\begin{figure}[h]
\centering
\includegraphics[width=\textwidth]{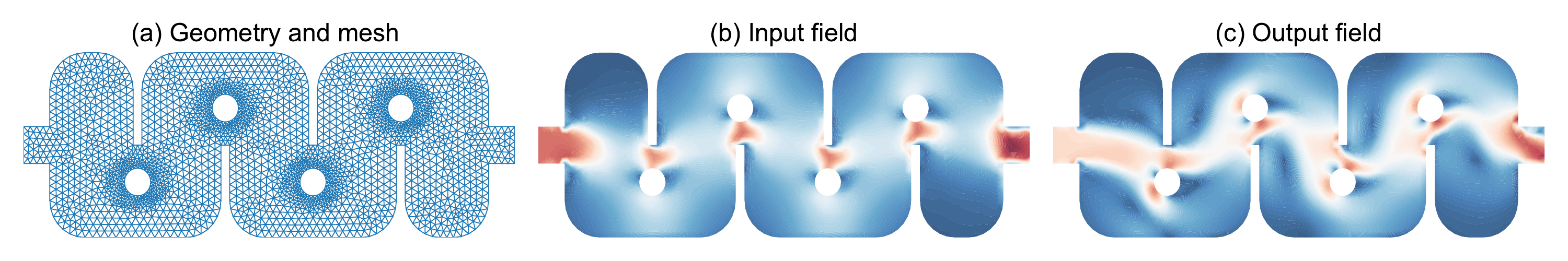}
\caption{The geometric domain and input-output pair of the pipe turbulence  case.}
\label{fig: pipeflow}
\end{figure}

\subsubsection{Heat transfer (Case 3)}
Heat transfer describes the transfer of energy as a result of a temperature difference, which widely exists in nature and engineering technology fields. A solid heat transfer case for a three-dimensional complex part is constructed to verify the ability of the method to handle the prediction problem with complex geometric domains. The heat equation can be represented in the following form (assuming no mass transfer or radiation). 

\begin{equation}
\label{heat_trans}
\rho C \frac{\partial T}{\partial t}= \nabla \cdot K\nabla T+\dot{Q}
\end{equation}
where $T$ is temperature as a function of time and space. $\rho$, $C$, and $K$ are the density, specific heat capacity, and thermal conductivity of the medium, respectively. And  $\dot{Q}$ is the internal heat source. For this case, the part material is copper with a residual resistivity ratio of 30, and $\dot{Q}$ is set to 0. The three-dimensional design model of the solid part is shown in Fig. \ref{fig: part}.

\begin{figure}[h]
\centering
\includegraphics[width=8cm]{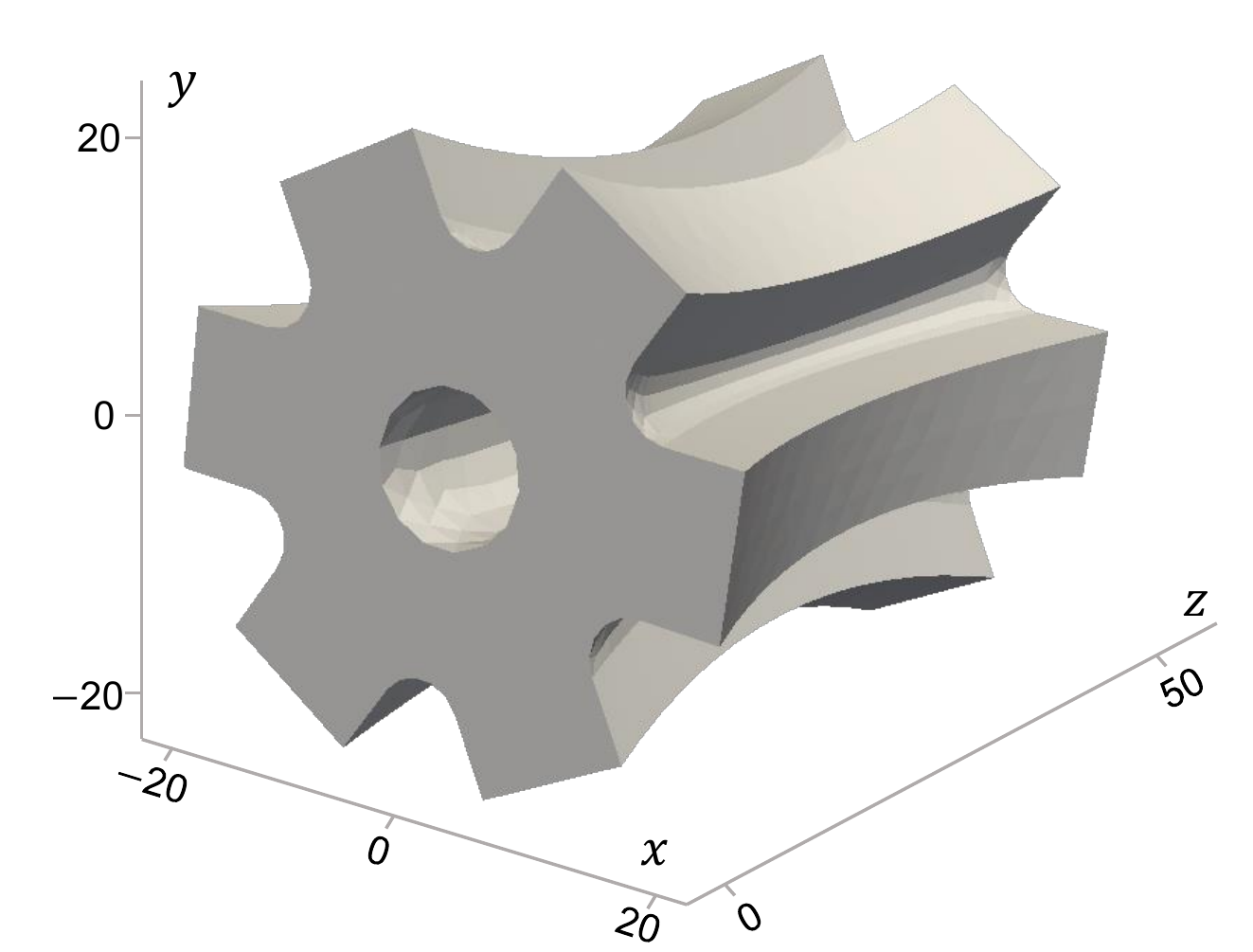}
\caption{The 3-dimensional model for the heat transfer case.}
\label{fig: part}
\end{figure}

The boundary conditions are imposed on the left and right sides. Specifically, on the left side $(z=0)$ is the low-temperature zone with the boundary condition $T_L (x,y,z=0)$. On the right side $(z=50)$ is the high-temperature zone with the boundary condition $T_H (x,y,z=50)$. The heat transfer problem is solved by the commercial simulation software Comsol, and the two boundary conditions are set as follows.

\begin{equation}
    T_L(x, y, z=0)=\left\{\begin{array}{l}T_1, 0<\theta \leq \pi/3 \\ 
                                          T_2, \pi/3<\theta \leq 2\pi/3 \\ 
                                          T_3, 2\pi/3<\theta \leq \pi \\ 
                                          T_4, \pi<\theta \leq 4\pi/3  \\
                                          T_5, 4\pi/3<\theta \leq 5\pi/3 \\ 
                                          T_6, 5\pi/3<\theta \leq 2\pi
                          \end{array}\right.
\end{equation}

\begin{equation}
    T_H(x, y, z=50)=\left\{\begin{array}{l}420,5<\sqrt{x^2+y^2} \leq 9 \\ 440,9<\sqrt{x^2+y^2} \leq 15 \\ 430, \sqrt{x^2+y^2}>15\end{array}\right.
\end{equation}
where $\theta$ is the center angle corresponding to the mesh node,  $T_1$, $T_2$, $T_3$, $T_4$, $T_5$, and $T_6$ are six temperature parameters which are randomly sampled from $290K-350K$.

The learning problem of this case is defined as the mapping from the low-temperature boundary condition $T_L (x,y,z=0)$ to the solid part's 3-dimensional temperature field $T_{t=3} (x,y,z)$ after 3s of heat transfer. The mesh for the input and output field is shown in Fig. \ref{fig: heat}a and Fig. \ref{fig: heat}c. The input domain is discretised by the triangular mesh with 186 nodes, and the output domain is discretised by the tetrahedral mesh including 7199 nodes. Furthermore, the input temperature field and output temperature field are shown in Fig. \ref{fig: heat}b and Fig. \ref{fig: heat}d. Note that the space domains for input and output are different in this case. The training data set consists of 100 labelled data, and another 100 samples are defined as test data.

\begin{figure}[h]
\centering
\includegraphics[width=\textwidth]{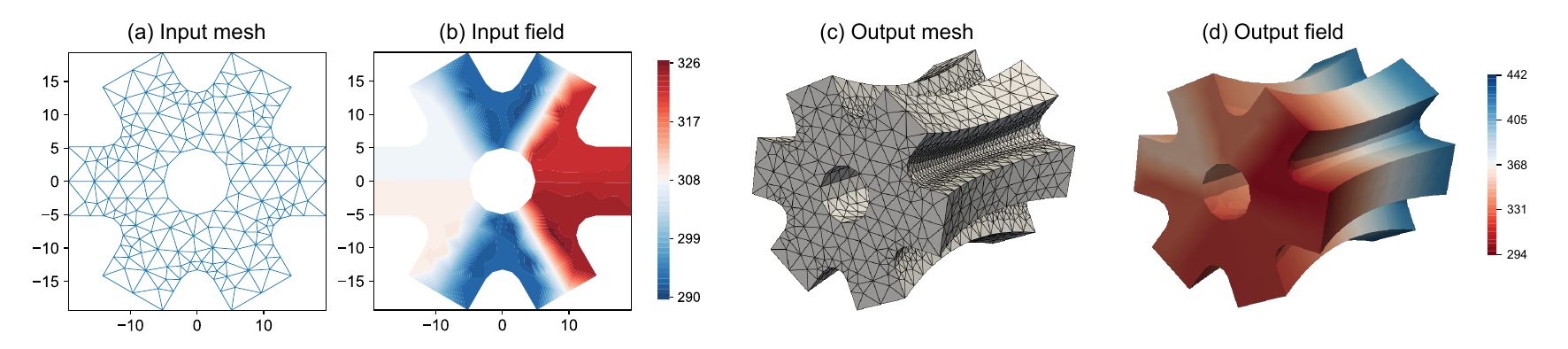}
\caption{The mesh and field for input and output data of the heat transfer case.}
\label{fig: heat}
\end{figure}

\subsection{Composite workpiece deformation prediction (Case 4)}

\subsubsection{Background}

Carbon Fiber Reinforced Polymer (CFRP) composite materials, which are lightweight and high-strength, are preferred materials for weight reduction and performance enhancement in modern aerospace industries \cite{liu2021multi}. CFRP parts used in aerospace have large size and complex shapes, therefore imposing higher requirements on deformation control during the manufacturing process. As one of the key processes of composites manufacturing, curing refers to using high temperatures to stimulate the chemical reactions and physical changes of the resin, thereby forming CFRP parts with load-bearing properties. Non-uniform residual stresses generated during the curing process can cause curing deformations such as spring-back, warpage, and bending-twisting combination, which not only risks the CFRP parts being scrapped but also becomes an important reason for damages and failures during subsequent assemblies \cite{ding2015three}.

Regulating the curing temperature distribution of a part is an effective means of controlling curing deformation. However, optimising the curing temperature field usually requires a large number of iterations based on the prediction results of the curing deformation field. Therefore, establishing a fast prediction model from the curing temperature field to the deformation field is of great significance for optimising and designing the temperature field of CFRP parts \cite{struzziero2019numerical}. Numerical simulation methods, such as the finite element method, have become the most widely used curing process modelling methods. However, high-fidelity curing deformation simulation requires accurate modelling of complex physicochemical processes and fine meshing of the part calculation domain, resulting in highly expensive and time-consuming calculations. Therefore, the computational efficiency of the traditional numerical modelling methods is insufficient to meet the requirements for the temperature field optimisation of the CFRP parts. Establishing a data-driven temperature-to-deformation prediction model can provide essential support for further curing process optimising. 

The CFRP workpiece used for verification is the air-intake structural part of a jet. As shown in Fig. \ref{fig:workpiece}, this workpiece is a complex closed revolving structure formed by multiple curved surfaces, which would deform significantly after high-temperature curing. The curing process is zoned self-resistance electric heating, where the internal and external surfaces of the workpiece are divided into multiple areas according to the radius of curvature for independent temperature control. The theoretical support of this case can be found in the authors' previous work \cite{liu2023active}.

\begin{figure}[h]
\centering
\includegraphics[width=0.8\textwidth]{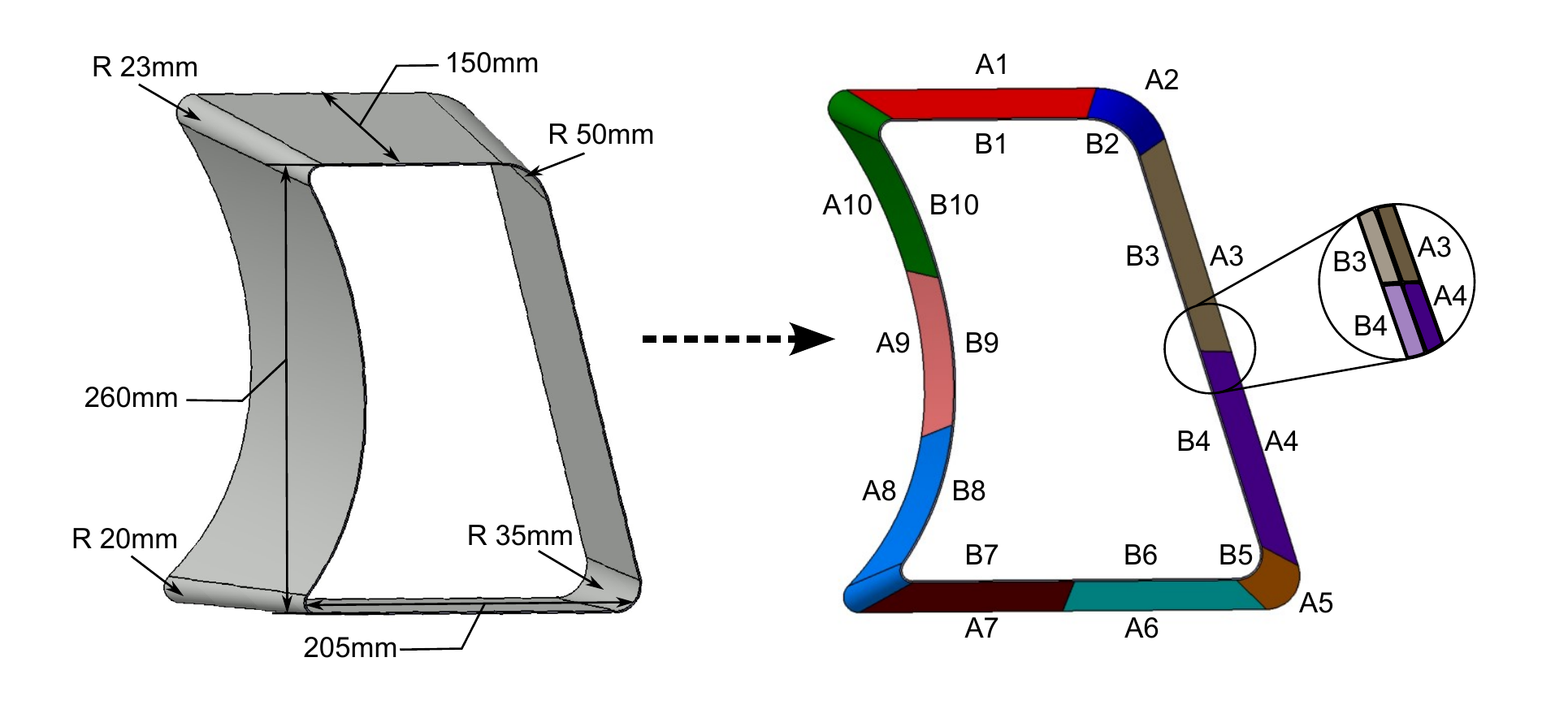}
\caption{The CFRP part for case study.}
\label{fig:workpiece}
\end{figure}

\subsubsection{Data generation}
The internal and external surfaces of the composite part are divided into 20 separate curing zones, with the temperature of each zone generated randomly between the $370K\sim400K$. The temperature fields and the deformation fields of the composite part were simulated by considering heat transfer, curing reactions, viscoelastic mechanics and other processes. As shown in Fig. \ref{fig:Composite}a, the part geometry is represented by a tetragonal mesh constructed in the commercial simulation software Comsol Multiphysics, comprising a total of 8232 nodes. A total of 500 data pairs of temperature-to-deformation fields were simulated, 400 of them are defined as training data and the rest 100 as test data. The examples of the input temperature field and output deformation field are shown in Fig. \ref{fig:Composite}b and Fig. \ref{fig:Composite}c. 

\begin{figure}[h]
\centering
\includegraphics[width=\textwidth]{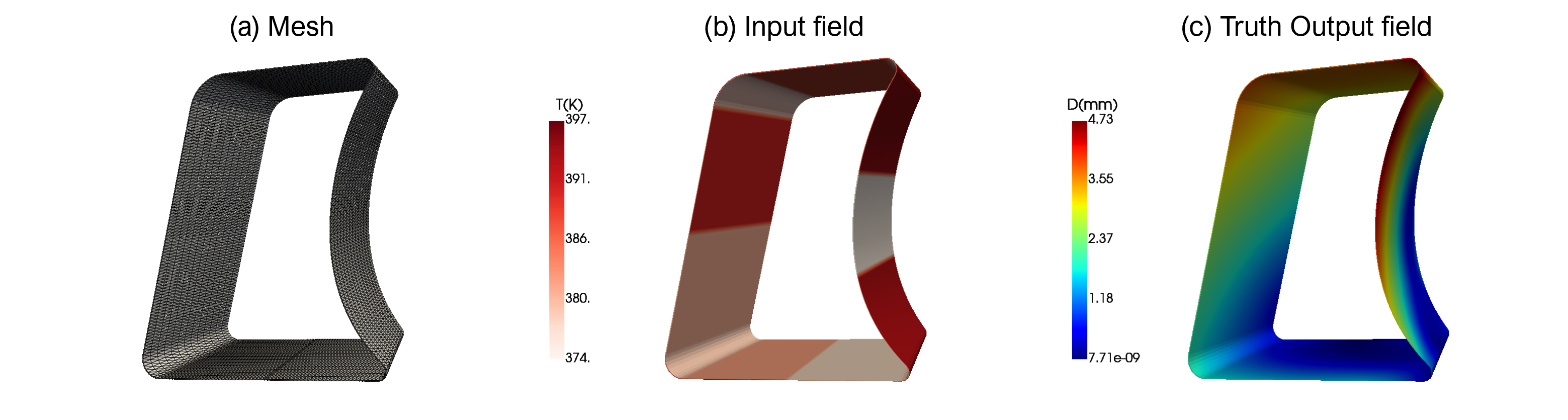}
\caption{The mesh and field for input and output data of the composite case.}
\label{fig:Composite}
\end{figure}

\subsection{Blood flow dynamics prediction (Case 5)}

\subsubsection{Background}
Blood flow dynamics is the science of studying the characteristics and regularities of the movement of blood and its constituents in the organism \cite{marsden2014optimization}. The driving force of hemodynamic research consists of the following three aspects: 1) hemodynamic research can assist researchers in studying the laws of blood flow in the vascular system of a healthy human body \cite{secomb2016hemodynamics}; 2) hemodynamic research can help analyse the causes and effects of cardiovascular diseases \cite{hope2010clinical,scotti2007compliant}; and 3) hemodynamic research can promote the development and optimisation of diagnosis and treatment techniques for vascular diseases from a therapeutic point of view \cite{marsden2014optimization}. In recent years, the development of measurement techniques has made it possible to reconstruct patient-specific vascular structures by CT imaging \cite{shahcheraghi2002unsteady} and 4D MRI \cite{markl20124d}. Furthermore, computational fluid dynamics (CFD) modelling has been used to simulate blood flow by numerically solving the Navier-Stokes equations, showing promising potential in clinical practice \cite{caballero2013review,morris2016computational}. On the one hand, CFD allows the non-invasive acquisition of haemodynamic parameters that in vitro measurements cannot measure. On the other hand, CFD can provide visualisation of the flow field results to investigate the effect of specific structures on haemodynamics. Despite the excellent predictive performance of CFD modelling, its high computational cost and the long processing time have prevented it from clinical practice in time-sensitive areas such as preoperative planning and serial monitoring \cite{kissas2020machine}. To address the above limitations, we aim to explore the possibilities of data-driven neural operator models for the surrogate modelling of haemodynamic CFD.

This case focuses on the hemodynamics of the human thoracic aorta, the largest human artery responsible for transporting oxygen and nutrient-rich blood to various organs. We consider a similar monitoring scenario in the paper \cite{maul2023transient}, where the inputs are the time-varying blood flow metrics monitored in real-time such as flow rate, blood flow, pressure, etc., at the inlet and outlet \cite{wen2010investigation}, and the outputs are the velocity field of the aorta. The field outputs can provide more information reflecting the state and evolution of the patient's disease than the individual metrics.

\begin{figure}[h]
\centering
\includegraphics[width=0.6\textwidth]{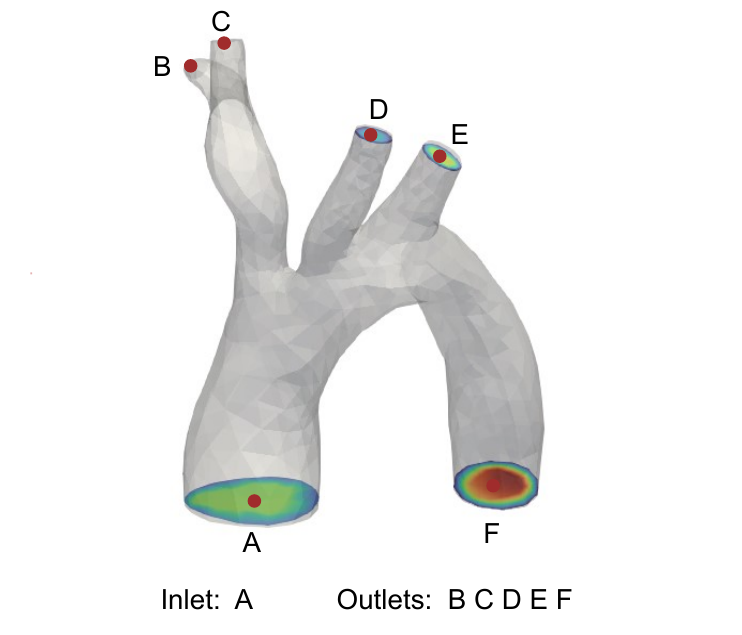}
\caption{The diagram of inlet and outlets of the aorta.}
\label{fig: InOutlet}
\end{figure}

\subsubsection{Data generation}
The aorta contains 1 inlet, i.e. the ascendens aortae, and 5 outlets, i.e. the descendens aortae, the left/right subclavian arteries, and the left/right common caroti arteries, as shown in Fig. \ref{fig: InOutlet}. The velocity boundary condition is imposed at the inlet, and pressure boundary conditions are imposed at the outlets, which describes the time-varying characteristic of velocity and pressure during one cardiac cycle (1.2 s). To simulate the gradual increase of velocity and pressure to a peak during the systolic phase and a fall back from the peak during the diastolic phase, we approximated the changes of velocity and pressure using a simplified Gaussian function, as shown in Fig. \ref{fig: PV_curves}. Then a set of boundary conditions can be determined by setting the mean, bandwidth, and peak values. 

Blood is assumed to be a homogeneous Newtonian fluid with a density of 1060 $kg/m^3$ and a viscosity of 0.0035 $N\cdot s/m^2$, and the flow of blood in the aorta is laminar flow. The vessel wall was assumed to be rigid, and no-slip conditions were considered. A total of 500 velocity/pressure curves are generated as boundary conditions, which are treated as inputs. Velocity fields are simulated based on Comsol Multiphysics software, which serves as outputs. The simulation results are derived using tetrahedral mesh (1656 spatial nodes), with a sampling interval of 0.01s in the time dimension, containing 121 time nodes. 400 of them are used as training data and the rest 100 are used as test data.

\begin{figure}[t]
\centering
\includegraphics[width=0.95\textwidth]{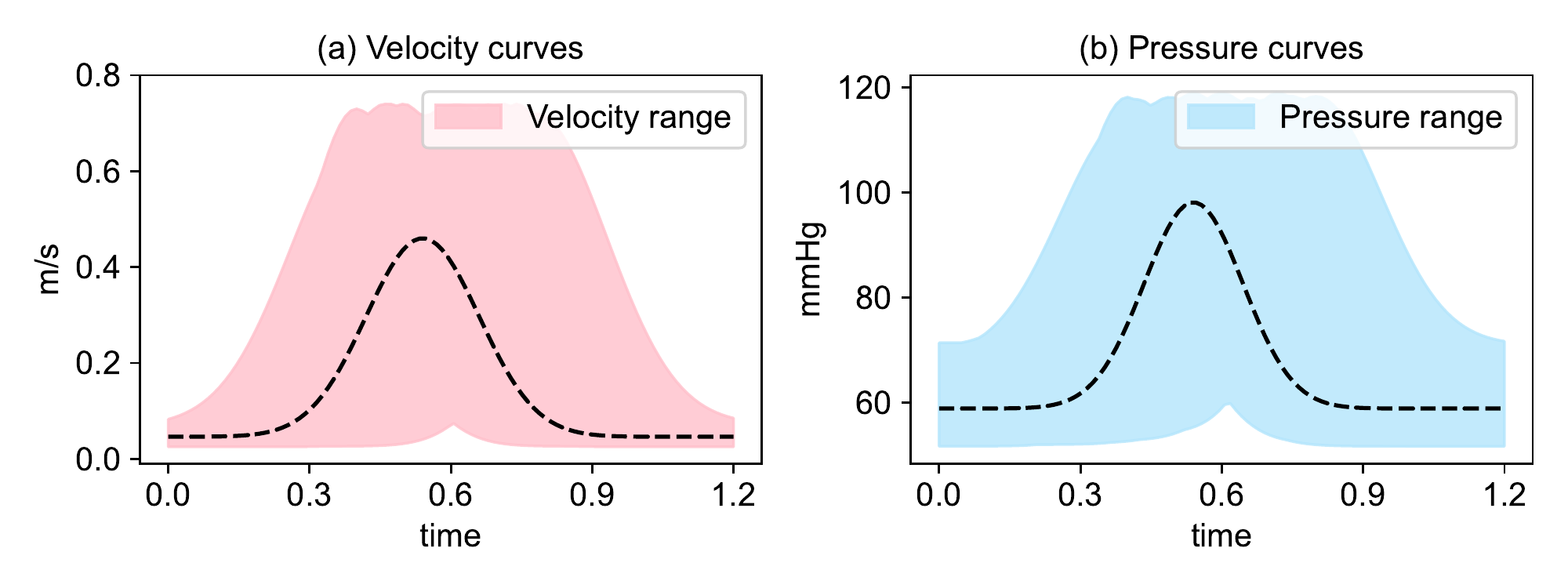}
\caption{The velocity and pressure curves.}
\label{fig: PV_curves}
\end{figure}

\clearpage
\section{Experimental setting}

 We compare the performance of the proposed \textbf{NORM} on the experimental cases with the existing representative operator learning methods \textbf{FNO}, \textbf{DeepONet}, \textbf{POD-DeepONet}, and a popular GNN architecture \textbf{GraphSAGE}. A brief introduction to each method is given below to help the reader understand the architecture and parameter setting of each model.
 
\textbf{NORM} can learn the mapping between functions on Riemannian manifolds. The proposed NORM consists of multiple encoder-approximator-decoder blocks, in which the encoder and the decoder are constructed as the spectral decomposition and the spectral reconstruction on the eigenfunctions of the Laplace–Beltrami operator (LBO). The LBO eigenfunctions of the geometric domain and its pseudo-inverse can be pre-calculated before training. Note that Darcy problem, Pipe turbulence and Composites cases adopt the structure of Fig. \ref{fig:3manifolds}a. Heat transfer case adopts the structure of Fig. \ref{fig:3manifolds}b. Blood flow case follows the Fig. \ref{fig:3manifolds}c, where the encoder and decoder of several ending L-layers employ both LBO eigenfunctions of $\cY$ and Fourier basis to process the output spatiotemporal functions.

\textbf{FNO} \cite{li2020fourier} parameterises the integral operators in the Fourier domain, and then the high-dimensional operator mapping can be transferred to the low-dimensional discretisation-invariant parameterisation of the few frequency modes. Since the original FNO cannot deal with the irregular geometric domain, the mesh interpolation solution from the paper \cite{lu2022comprehensive} is adopted to construct a regular mesh for the FNO. And the final prediction error is calculated on the original irregular grid by the second interpolation from the regular grid to the irregular grid.  Considering the prohibitive computational burden of the spatial mesh interpolation of 3D parts, FNO is only implemented in two cases: Darcy flow and pipe turbulence. And the interpolation resolution ratios are 101*101 and 32*128 for the Darcy flow and pipe turbulence cases, respectively.

\textbf{DeepONet} \cite{lu2021learning} is a neural operator framework based on the universal approximation theorem. The branch net of DeepONet encodes the input function, and another trunk net encodes the grid coordinates to be queried for the output function. The combination of the two networks enables the function output that can provide the prediction result of any point in the domain.

\textbf{POD-DeepONet} \cite{lu2022comprehensive} is the latest variant of DeepONet, in which Proper Orthogonal Decomposition (POD) is performed on the training data to compute the bases for output data. The POD bases are used as the trunk net (The POD basis can be precomputed before training, no training required.), and the branch network can directly learn the weights of POD bases.

\textbf{GraphSAGE} \cite{hamilton2017inductive} is a popular GNN architecture that uses SAGE convolutions, which is an inductive learning framework that can utilize the attribute information of the vertex to effectively generate the unknown vertex embedding.

Due to the different learning problems of each case, it is difficult to adopt an exact same set for each methods. The detailed architecture and parameters setting of each method for different cases are summarised in Table \ref{tab_para}, in which $d_m$ and $d_t$ mean the number of LBO/POD basis and Fourier basis, respectively. $d_v$ denotes the channel number after the mapping $\cP$.

\begin{table}[h]%
\scriptsize
\vspace*{-2pt}
\centering
\caption{ The model setting of different methods in all experiment cases.} \label{tab_para}%
\renewcommand\arraystretch{1.4}
\setlength{\tabcolsep}{3.5pt}
\begin{tabular*}{1\textwidth}{ll lllll}
\toprule
\textbf{Methods} & \textbf{Setting} & \textbf{Darcy problem}  & \textbf{Pipe turbulence } & \textbf{Heat transfer}  & \textbf{Composite}     & \textbf{Blood flow} \\
\midrule

\multirow{3}{4em}{Data size}    
& Train data  &1000     & 300    &100     & 400   & 400 \\
& Test data   &200      & 100    &100     & 100  &  100 \\
& Batch size  &100      & 50     &10      & 20  &  10   \\
\midrule

\multirow{4}{4em}{GraphSage}    
& SAGE-Convs       & 4     &   2         &  -   & 4        & - \\
& Hidden features  & 32     &   16        &  -   & 64        & - \\
& Linear-Layers    & 32*32*1     &   16*16*1   &  -   & 64*64*64*1        & - \\
& Epoches          & 2000     &   500       &  -   & 2000        & - \\
\midrule

\multirow{3}{4em}{DeepONet}    
&Branch net   & 256*256*100      & 256*256*32 & 256*256*100      & 256*256*100  &   256*256*256  \\
&Trunck net   & 128*128*128*100  & 128*128*32 & 128*128*128*100  & 128*128*128*100   &   256*256*256*256 \\
&Epoches      & 5000             & 3000       & 5000             & 5000   &  1000  \\
\midrule

\multirow{3}{4em}{POD-DeepONet}    
&$d_m$         &64     & 128   & 64  & 128  &  64\\
&Branch net    &Ref\cite{lu2022comprehensive}       &  Ref\cite{lu2022comprehensive}    & Ref\cite{lu2022comprehensive}    & Ref\cite{lu2022comprehensive} & 512*512\\
&Epoches       &5000   & 1000 & 5000  & 5000   &  1000  \\
\midrule

\multirow{4}{4em}{FNO}    
&$d_m$    &[20,20]    & [16,16] & -  & -  &-  \\
&$d_v$    &32         & 32  & -   &-  &- \\
&F-Layers &4          & 4   & -   & - &- \\
&Epoches  &1000       & 1000 &- & -   &-    \\
\midrule

\multirow{4}{4em}{NORM}    
&$d_m$   &128    & 128  & 128  & 128  &  64\\
&$d_t$   &-    & -  & -  & -  &  16\\
&$d_v$   &32     & 32   & 64    & 32  &  16  \\
&L-layers  &4     & 4   & 4     & 4   &   5  \\
&Epoches   &1000  & 1000& 2000  & 2000   &  500  \\

\bottomrule
\end{tabular*}\vspace*{-3pt}  
\label{tab: tab_para}
\end{table}

\clearpage
\section{Supplementary experimental results}

\begin{figure}[h]
\centering
\includegraphics[width=\textwidth]{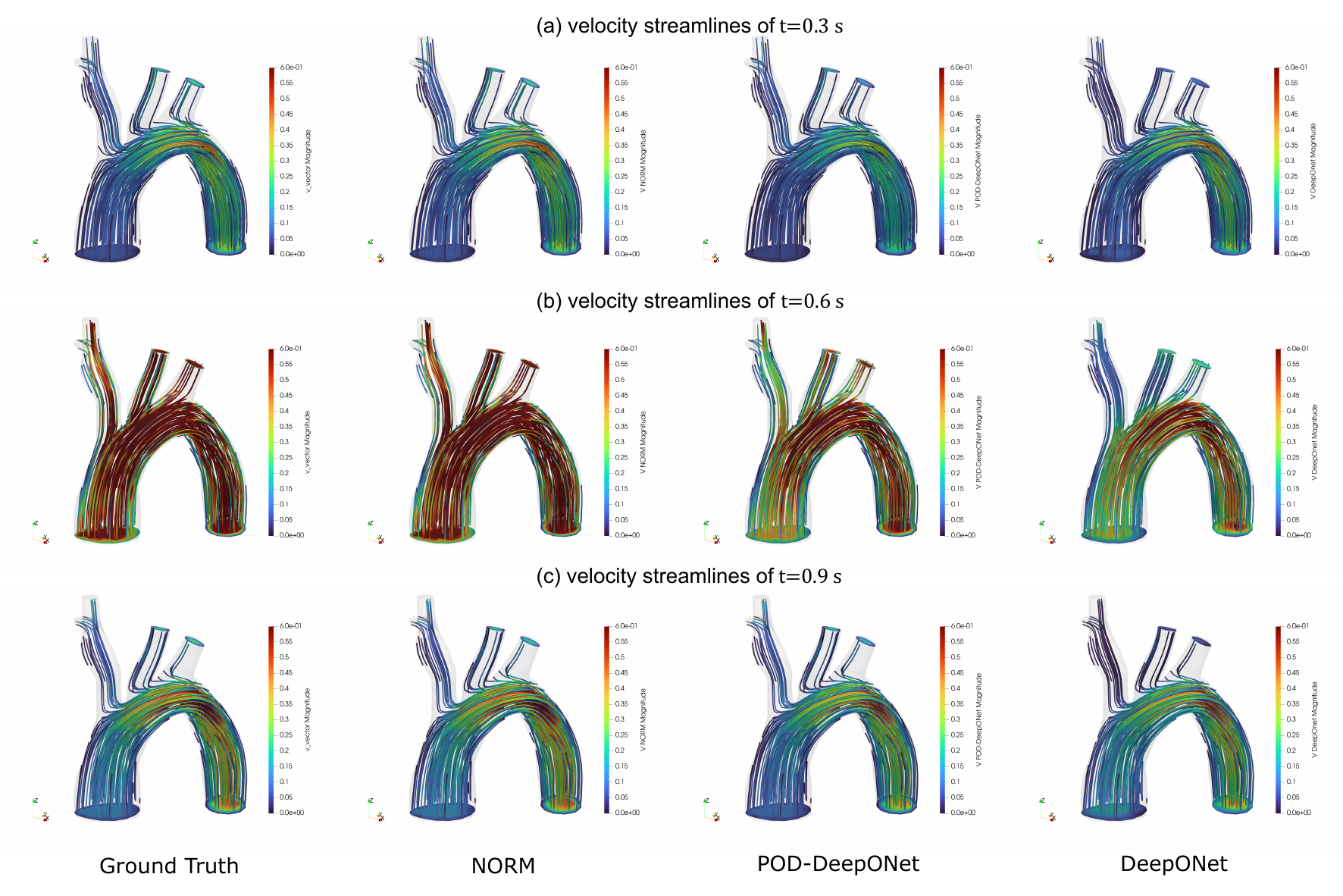}
\caption{Comparison visualisations of velocity field at different moments.}
\label{fig: Supp_blood_results}
\end{figure}

\newpage

\FloatBarrier

\bibliographystyle{unsrt}

\bibliography{supplement}